\documentclass[11pt]{amsart}

\usepackage[leqno]{amsmath}
\usepackage{amssymb}
\usepackage{amsxtra}
\usepackage{amscd}
\usepackage{hyperref}
\usepackage{graphicx, color}
\usepackage{dcpic}

\newtheorem{thm}{Theorem}
\newtheorem{prop}[thm]{Proposition}

\newtheorem{defn}[thm]{Definition}

\newtheorem{lemma}[thm]{Lemma}
\newtheorem{cor}[thm]{Corollary}

\newcommand{\I}{\ensuremath{\mathbb{I}}}

\newcommand{\Z}{\mathbb{Z}}

\newcommand{\graph}[1]{\Gamma_{L}}
\newcommand{\circles}[1]{\ensuremath{\mathrm{\textsc{Circles}}(#1)}}

\newcommand{\lra}{\longrightarrow}
\newcommand{\inlinediag}[2][0.33]{\includegraphics[scale=#1]{./#2}}
\newcommand{\myfig}[3][0.5]{\begin{center}\begin{figure}\includegraphics[scale=#1]{./#2}\caption{#3}\label{fig:#2}\end{figure}\end{center}}

\newcommand{\cross}[1]{\ensuremath{\mathrm{\textsc{cr}}(#1)}}

\newcommand{\resolution}[1]{\ensuremath{\mathrm{\textsc{res}}(#1)}}

\newcommand{\lefty}[1]{\overleftarrow{#1}}
\newcommand{\righty}[1]{\overrightarrow{#1}}

\newcommand{\cleaved}[1]{\mathcal{C\!L}_{#1}}
\newcommand{\mcleaved}[1]{\mathcal{MC\!L}(#1)}

\newcommand{\leftComplex}[1]{\langle\!\!\!\langle\,#1\,]\!]} 
\newcommand{\rightComplex}[1]{[\![\, #1 \,\rangle\!\!\!\rangle}

\newcommand{\Oz}{Ozsv\'{a}th}

\title{Planar algebras and the decategorification of bordered Khovanov homology}
\author{Lawrence P. Roberts}

\begin{document}
\begin{abstract}    
 We give a simple, combinatorial construction of a unital, spherical, non-degenerate $\ast$-planar algebra over the ring $\Z[q^{1/2},q^{-1/2}]$.  This planar algebra is similar in spirit to the Temperley-Lieb planar algebra, but computations show that they are different. The construction comes from the combinatorics of the decategorifications of the type A and type D structures in the author's previous work on bordered Khovanov homology \cite{Deca}.  In particular,  the construction illustrates how gluing of tangles occurs in the bordered Khovanov homology (\cite{RobA}) and its difference from that in Khovanov's tangle homology, \cite{Khta}, without being encumbered by any extra homological algebra. It also provides a simple framework for showing that these theories are not related through a simple process, thereby confirming recent work of A. Manion, \cite{Mani}. Furthermore, using Khovanov's conventions and a state sum approach to the Jones polynomial, we obtain new invariant for tangles in $\Sigma \times [-1,1]$  where $\Sigma$ is a compact, planar surface with boundary, and the tangle intersects each boundary cylinder in an even number of points. This construction naturally generalizes Khovanov's approach to the Jones polynomial.  \\
\end{abstract}
\maketitle

\section{Introduction}

\noindent Let $S$ be an oriented two-sphere.

\begin{defn}[Disc configurations]\label{defn:discconfiguration}
A \em{disc configuration} in $S$ is a non-empty ordered set $\mathbb{D} = (\lefty{D}_{0}; \lefty{D}_{1}, \ldots, \lefty{D}_{m})$, $m \geq 0$, of closed discs embedded in $S$, and a choice of points $\ast_{\lefty{D}_{i}} \in \partial \lefty{D}_{i}$  for each $i=0, \ldots, m$, such that
\begin{enumerate}
\item $\lefty{D}_{i} \subset \mathrm{int}\lefty{D}_{0}$ for each $i \geq 1$, and
\item $\lefty{D}_{i} \cap \lefty{D}_{j} = \emptyset$ for $i,j \geq 1$
\end{enumerate}
\end{defn}

\noindent An example of a disc configuration is given in Figure \ref{fig:discdiagram}.\\
\ \\
\noindent Each disc $\lefty{D}_{i}$ in $\mathbb{D}$ inherits an orientation from $S$, and $\partial\lefty{D}_{i}$ is oriented as the boundary $\lefty{D}_{i}$.  The discs $\lefty{D}_{i}$ will be called {\em inside discs}. The  {\em outside discs} for $\mathbb{D}$ are the discs $\righty{D}_{i} = S \backslash \big(\mathrm{int} \lefty{D}_{i}\big)$. \\
\ \\
\noindent Given a disc configuration $\mathbb{D}$ we identify an oriented sub-surface of $S$
$$\Sigma_{\mathbb{D}} = \lefty{D}_{0} \backslash \left(\bigcup_{i=1}^{m} \mathrm{int}\,\lefty{D}_{i}\right)$$

\begin{center}
\begin{figure}
\includegraphics[scale=0.75]{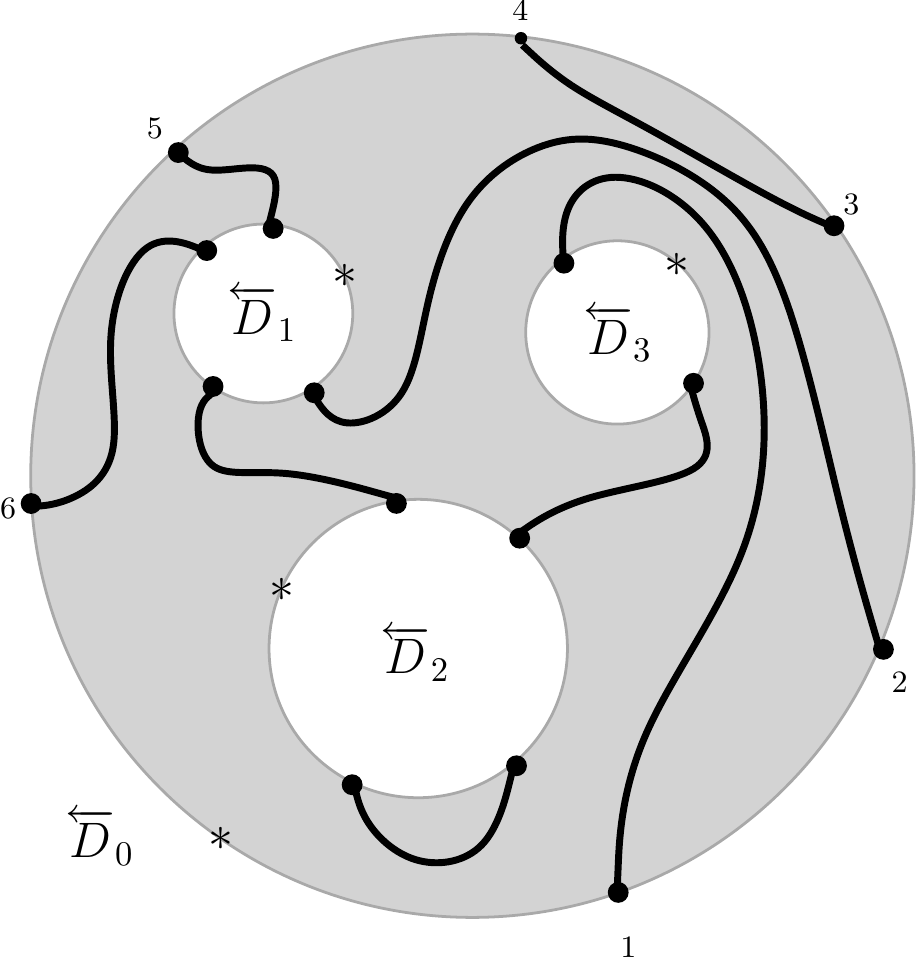}
\caption{A planar diagram with signature $(3;2,2,1)$ subordinate to the disc configuration $\mathbb{D}$. $\Sigma_{\mathbb{D}}$ is the shaded planar surface. The points in $\partial D_{0}$ have been labeled according to their ordering.}\label{fig:discdiagram}
\end{figure}
\end{center}

\begin{defn}[Planar diagram]\label{defn:planardiagram}
A planar diagram $P$ subordinate to a disc configuration $\mathbb{D}$ is a set of arcs and circles properly and disjointly embedded in $\Sigma_{\mathbb{D}}$ such that each boundary circle of $\Sigma_{\mathbb{D}}$ intersects the arcs of $P$ in an {\em even} number of points.
\end{defn}

\noindent Given $P$ let $\partial\lefty{D}_{i} \cap P$ be ordered according to the orientation of $\partial\lefty{D}_{i} \backslash\{\ast_{\lefty{D}_{i}}\}$. the signature of $P$, denoted $\textsc{Sign}(P)$, is the tuple $(n_{0}; n_{1}, \ldots n_{m})$ where $2n_{i} = |\partial\lefty{D}_{i} \cap P|$. The diagram in Figure \ref{fig:discdiagram} is a planar diagram. \\
\ \\
\noindent To an oriented circle with a marked point $\ast$ and $2n$ labeled points, we will associate a free $\Z[q^{1/2},q^{-1/2}]$-module $\mathcal{I}_{2n}$. Our goal is to define for any planar diagram $P$  subordinate to $\mathbb{D}$, with $\textsc{Sign}(P)=(n_{0}; n_{1}, \ldots, n_{m})$, a linear map
$$
Z_{P} :  \mathcal{I}_{2n_{1}}\!\otimes\!\mathcal{I}_{2n_{2}}\!\otimes \cdots \otimes\!\mathcal{I}_{2n_{m}} \lra \mathcal{I}_{2n_{0}}
$$
where the tensor products are taken over $\Z[q^{1/2}, q^{-1/2}]$. \\
\ \\
\noindent  The module $\mathcal{I}_{2n}$ comes from the idempotents of the differential bigraded algebra $\mathcal{B}\Gamma_{2n}$ used in bordered Khovanov homology \cite{RobA}, \cite{RobD}, and which supports the decategorification of differential bigraded modules over $\mathcal{B}\Gamma_{2n}$, \cite{Deca}. The maps $Z_{P}$ are a combinatorial generalization to planar surfaces of the combinatorics of decategorification in \cite{Deca}. Our motivation, besides the simplicity of the construction, is 1) to use these maps to compare the constructrions of bordered Khovanov homology to Khovanov's tangle homology, and 2) to give a simple generalization of the Jones polynomial to tangles with good compositional properties, that is different from those the author could find in the literature.  \\
\ \\
\subsection{The modules $\mathcal{I}_{2n}$} To an oriented circle $C$, embedded in an oriented sphere $S$, with a marked point $\ast \in C$, and a subset $Q \subset C \backslash\{\ast\}$ of $2n \geq 0$ points, we assign the free $\Z[q^{1/2}, q^{-1/2}]$-module generated by the elements of the following set:

\begin{defn}[Cleaved links]
$\cleaved{n}$ is the set of equivalence classes of pairs $(L,\sigma)$ where
\begin{enumerate}
\item $L$ is a link embedded in $S$ so that each component of $L$ transversely intersects $C$,
\item $L \cap C = Q$
\item $\sigma: \circles{L} \longrightarrow \{+,-\}$
\end{enumerate}
where two such pairs are equivalent if there is an orientation preserving diffeomorphism of $S$ which maps the equator to the equator, carries $L_{1}$ to $L_{2}$, and preserves the marked points and decorations on each component. 
\end{defn} 

\noindent The elements of $\bigcup_{n=0}^{\infty} \cleaved{n}$ will be called {\em cleaved links} and the maps $\sigma$ will be called decorations. We note that $\cleaved{0}$ is the empty link in a sphere with an marked equator. 

 \begin{defn}
 For each $n \geq 0$, $\mathcal{I}_{2n}$ is the free $\Z[q^{1/2}, q^{-1/2}]$-module generated by the elements of $\cleaved{n}$. The generator corresponding to $(L,\sigma) \in \cleaved{n}$ will be denoted $I_{(L,\sigma)}$. 
 \end{defn}

 \noindent{\bf Examples:} When $n=0$, there is only one equivalence class $I_{0}$ of cleaved links, so $\mathcal{I}_{0} =  \Z[q^{1/2}, q^{-1/2}] I_{0}$. The generators of $\mathcal{I}_{2}$ are cleaved links with one component $K$  which intersects the equator $C$ exactly twice.  $K$ can be decorated with either a $+$ or a $-$, corresponding to the two generators $I_{+}$ and $I_{-}$ for  $\mathcal{I}_{2}$.  Thus,  $\mathcal{I}_{2} \cong \Z[q^{1/2},q^{-1/2}] I_{+} \oplus \Z[q^{1/2},q^{-1/2}] I_{-}$. For $\mathcal{I}_{4}$, there are twelve generators corresponding to the elements in $\cleaved{2}$, depicted and labeled in Figure \ref{fig:I4Gens}.

\myfig{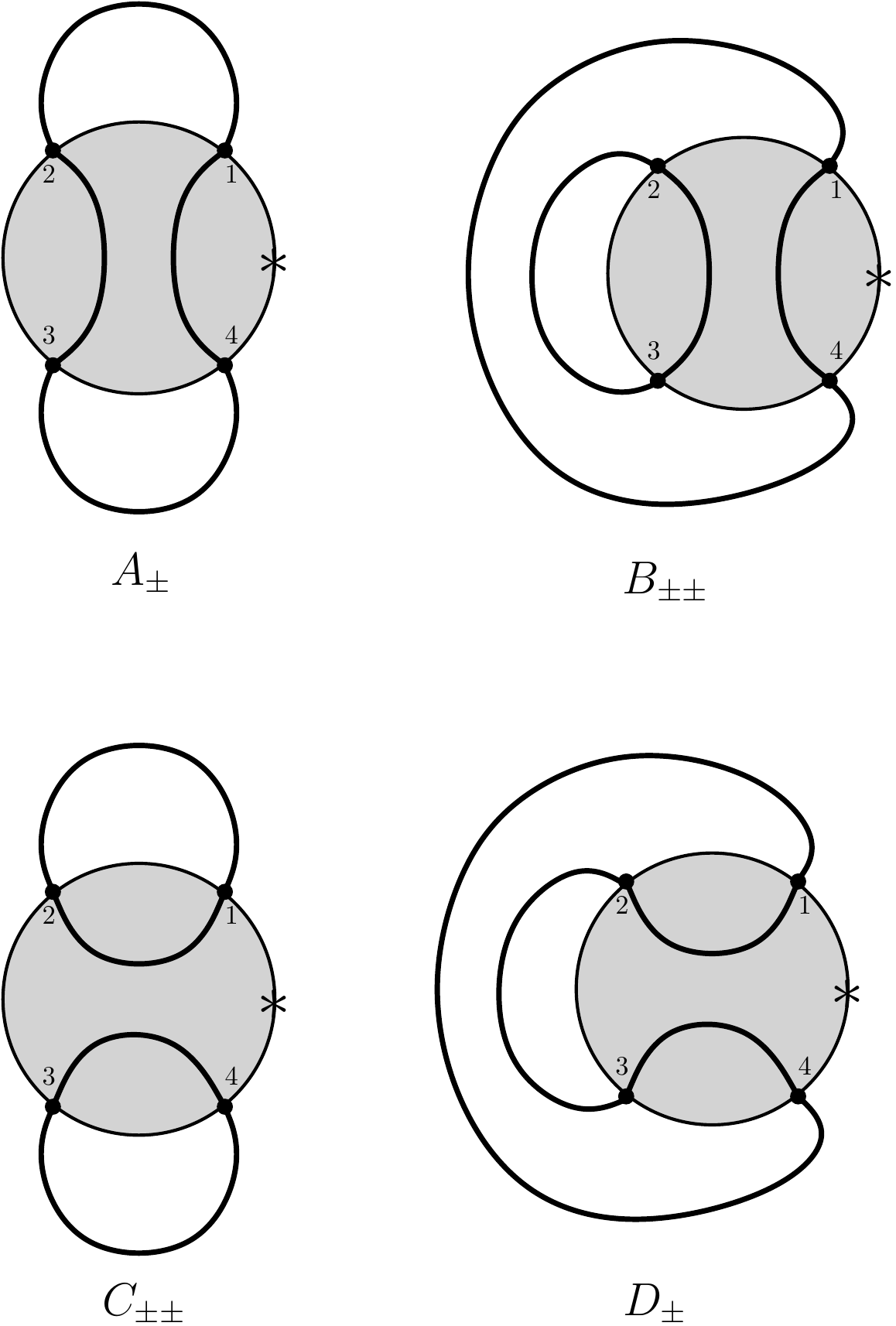}{The twelve generators of $\mathcal{I}_{4}$ grouped based on their inside and outside matchings. There a two generators of type $A$ and $D$, and four of type $B$ and $C$, as determined by the choice of decorations.}

\subsection{The maps $Z_{P}$} To define the map $Z_{P}$ we describe it on the generators of $\mathcal{I}_{2n_{1}} \otimes \mathcal{I}_{2n_{2}} \otimes \cdots \otimes \mathcal{I}_{2n_{m}}$ and then extend linearly. We start with some preliminaries.

\begin{defn}[Multiply Cleaved Links]\label{defn:multicleaved}
A decorated, multiply cleaved link $(M,\sigma)$ subordinate to a disc configuration $\mathbb{D}$ in $S$ is a {\rm (}possibly empty{\rm )} embedded submanifold $M \subset S$, consisting of finitely many circles, $\circles{M}$, and a {\em decoration} map $\sigma: \circles{M} \lra \{+,-\}$, such that  
\begin{enumerate}
\item\label{item:rightyi} $M$ intersects the boundary $\partial\lefty{D}_{i}$ of each disc $\lefty{D}_{i}$ transversely, and away from $\ast_{\lefty{D}_{i}}$, and 
\item\label{item:rightyii} $M$ has no component contained in $\mathrm{int}\righty{D}_{0}$, or in $\mathrm{int}\lefty{D}_{i}$ for $i=1, \ldots, m$. 
\end{enumerate} 
\end{defn}

\noindent Each multiply cleaved link is built on top of a planar diagram:

\begin{defn}\label{defn:assocplanar}
For a decorated, multiply cleaved link $(M,\sigma)$ subordinate to $\mathbb{D}$, $\mathfrak{P}(M)$ is the associated planar diagram $M \cap \Sigma_{\mathbb{D}}$. 
\end{defn}

\noindent  Then $\mathfrak{P}(M)$ is subordinate to $\mathbb{D}$ and inherits the orderings and marked points of $M$. For example, given the multiply cleaved link $(M,\sigma)$ in Figure \ref{fig:multi}, $\mathfrak{P}(M,\sigma)$ is the planar diagram in Figure \ref{fig:discdiagram}.  We use $\mathfrak{P}(M)$ to import all the notions defined for planar diagrams to multiply cleaved links.

\begin{center}
\begin{figure}
\includegraphics[scale=0.75]{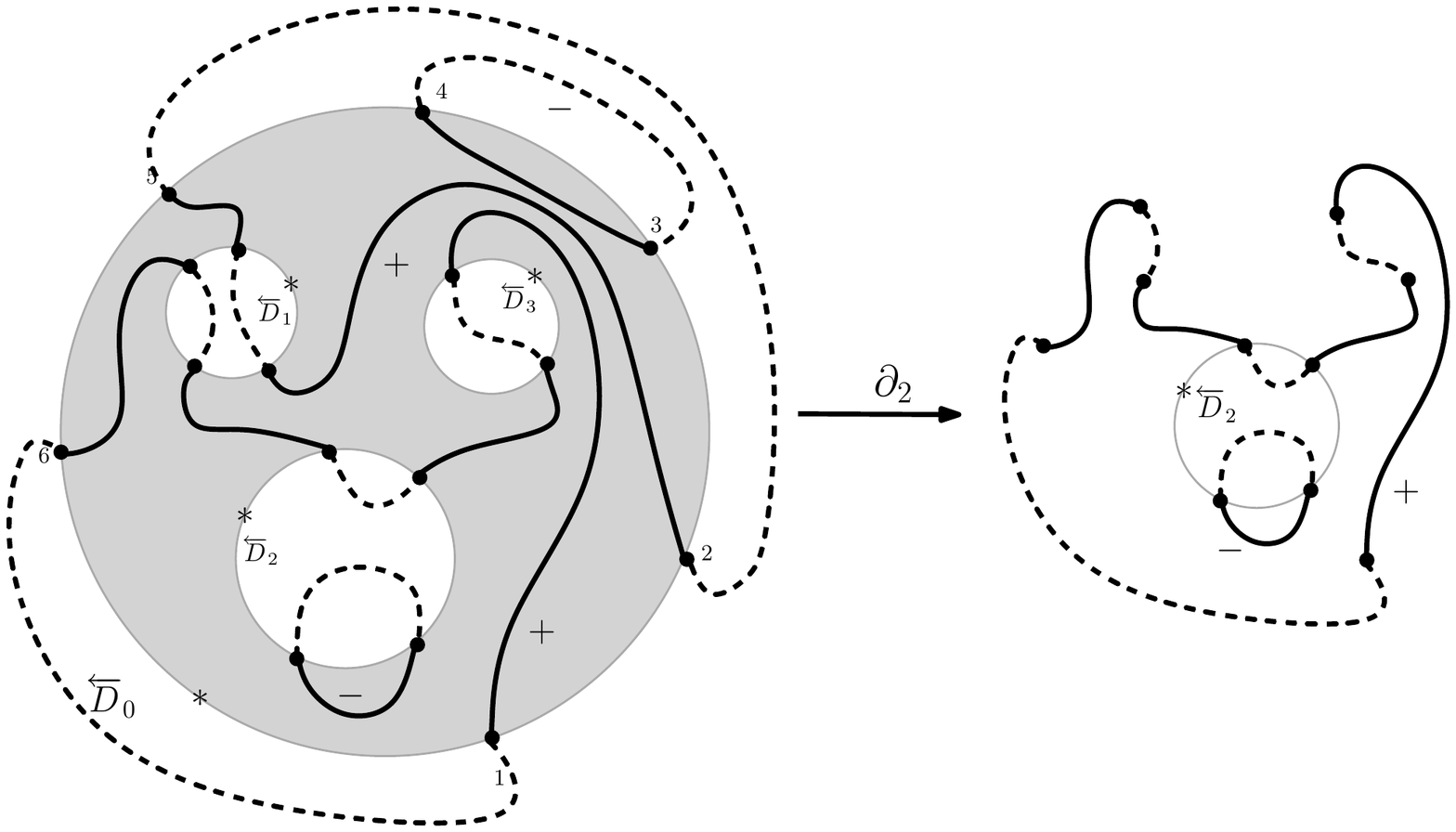}
\caption{The left diagram is a decorated, multiply cleaved link $(M,\sigma)$ with $\mathfrak{P}(M,\sigma)$ equal to the planar diagram in Figure \ref{fig:discdiagram}. The inside and outside matchings are shown as dashed lines. The $2^{nd}$ boundary $\partial_{2}(M,\sigma)$ is shown on the right.} \label{fig:multi}
\end{figure}
\end{center}

\begin{defn}
Let $(M,\sigma)$ be a decorated, multiply cleaved link. The set of equivalence classes of decorated, multiply cleaved links $(M,\sigma)$ with signature $(n_{0};n_{1}\!,\ldots,\!n_{m})$ will be denoted $\mcleaved{n_{0};n_{1}\!,\ldots,\!n_{m}}$.
\end{defn}

\noindent The advantage to using multiply cleaved links is that the have ``boundaries'':
\begin{defn} 
For each $i = 0, \ldots, m$, the $i^{th}$-boundary map $$\partial_{i} : \mcleaved{n_{0};n_{1}\!,\ldots,\!n_{m}} \lra \cleaved{n_{i}}$$  takes $(M,\sigma)$ to the cleaved link $\partial_{i}(M,\sigma)$ defined by
\begin{enumerate}
\item $\partial_{i}(M,\sigma)$ is subordinate to $(\lefty{D}_{i})$,
\item $\ast_{\partial_{i}(M,\sigma)} = \ast_{\lefty{D}_{i}}$
\item $\circles{\partial_{i}(M,\sigma)}$ is the subset of  the components of $M$ which intersect $\partial\lefty{D}_{i}$,
\item $\sigma_{\partial_{i}(M,\sigma)}$ is the restriction of $\sigma$ to $\circles{\partial_{i}(M,\sigma)}$
\end{enumerate}
\end{defn}

\noindent  For example, the $2^{nd}$ boundary of the multiply cleaved link on the left of Figure \ref{fig:multi} is the diagram on the right of that figure. \\
\ \\

\begin{defn}
Let $P$ be a planar diagram subordinate to $\mathbb{D}$. Then 
$$
\mathcal{M}(P) = \big\{\,(M,\sigma) \in \mcleaved{\textsc{Sign}(P)}\,\big|\,\mathfrak{P}(M) = P\, \big\}
$$ 
If we are given $(L_{i}, \sigma_{i}) \in \cleaved{n_{i}}$ for $i=0,\ldots, m$, then
$$
\mathcal{M}_{P}\big((L_{0}, \sigma_{0}); (L_{1},\sigma_{1}), \ldots, (L_{m},\sigma_{m})\big) = \big\{\,(M,\sigma)\in\mathcal{M}(P)\,\big|\, \partial_{i}(M,\sigma) = (L_{i}, \sigma_{i})\,\big\}
$$
\end{defn} 

\noindent  \noindent We will usually shorten this to $\mathcal{M}_{P}(L_{0}; L_{1}, \ldots, L_{m})$ with the understanding that there are decoration maps $\sigma_{i}$ on each of the cleaved links. The signature and the associated planar diagram do not depend on the decorations. The decorations, however, return in the {\em weight} of the link:

\begin{defn}
The weight $W(M,\sigma) \in \Z[q^{1/2},q^{-1/2}]$ of a decorated, multiply cleaved link $(M,\sigma)$ is the quantity 
$$
W(M,\sigma) = \left(\prod_{\footnotesize\begin{array}{c} C \in \circles{M}\\ \sigma(C) = + \end{array}\normalsize} q^{1 - \frac{1}{2} N_{M}(C)} \right)\cdot \left(\prod_{\footnotesize\begin{array}{c} C \in \circles{M}\\ \sigma(C) = - \end{array}\normalsize} q^{-1 + \frac{1}{2} N_{M}(C)} \right)
$$
where $N_{M}(C)$ is the number of circles $\partial \lefty{D}_{i}$, $i \geq 0$, which non-trivially intersect the circle $C \in \circles{M}$. If $\circles{M} = \emptyset$ then $W(M,\sigma) = 1$. 
\end{defn}

\noindent We are now in a position to define $Z_{P}$. For a generator 
$$
\xi = I_{(L_{1}, \sigma_{1})} \otimes I_{(L_{2}, \sigma_{2})} \otimes \cdots \otimes I_{(L_{m}, \sigma_{m})} \in \mathcal{I}_{2n_{1}} \otimes \mathcal{I}_{2n_{2}} \otimes \cdots \otimes \mathcal{I}_{2n_{m}}
$$
with $(L_{i}, \sigma_{i}) \in \cleaved{n_{i}}$ let
$$
Z_{P}(\xi) = \sum_{\footnotesize\begin{array}{c} (M,\sigma) \in \mathcal{M}(P) \\  \partial_{i}(M,\sigma) = (L_{i},\sigma_{i}), i \geq 1  \end{array}\normalsize} W(M,\sigma)I_{\partial_{0}(M,\sigma)}  
$$
\noindent Recall that $\mathcal{I}_{0} = \Z[q^{1/2},q^{-1/2}]$, so this definition applies even if $P$ does not intersect some of the boundaries $\partial\lefty{D}_{i}$ for $i \geq 0$. \\
\ \\
\noindent {\bf Warning:}  Changing the marked point on $\partial \lefty{D}_{i}$ can change the identification with the generators of $\mathcal{I}_{2n_{i}}$, and thus change the map $Z_{P}$. An example below illustrates this phenomenon.

\subsection{Planar algebra structure}

\noindent Planar diagrams can be composed, and we will describe how the partition maps behave under this composition. Suppose $T$ and $R$ are planar diagrams subordinate to disc configurations $\mathbb{D}_{T}$ and $\mathbb{D}_{R}$ with signatures $\textsc{Sign}(T) = (n_{0}; n_{1}, \ldots, n_{m})$ and $\textsc{Sign}(R) = (n'_{0}; n'_{1}, \ldots, n'_{m})$ such that $n_{0} = n'_{i}$ for some $i\in \{1, \ldots, m'\}$. Then there is a planar diagram $R \circ_{i} T$  subordinate to
$$
\mathbb{D}_{R} \circ_{i} \mathbb{D}_{T}=(\lefty{D}_{0}'; \lefty{D}'_{1},\ldots,\lefty{D}'_{i-1}, \lefty{D}_{1}, \ldots, \lefty{D}_{m}, \lefty{D}_{i+1},\ldots, \lefty{D}'_{m'})
$$ 
obtained by gluing the discs $S_{T} \backslash \mathrm{int}\righty{D}_{0}$ and $S_{R} \backslash \mathrm{int}\lefty{D}'_{i}$  along their boundaries, so that $\ast_{\lefty{D}_{0}}$ is glued to $\ast_{\lefty{D}'_{i}}$, and gluing induces and ordered bijection $\partial \lefty{D}_{0}\cap T \lra \partial \lefty{D}_{i'} \cap R$. 
\ \\
\noindent The definition of planar algebras in \cite{Jon1}, \cite{Jon2} provides a composition $Z_{R} \circ_{i} Z_{T}$ of $Z_{R}$ and $Z_{T}$ along the $i^{th}$ entry in the tensor product of the domain of $Z_{R}$. In particular, if $a_{i} \in \mathcal{I}_{2n'_{i}}$ and $b_{j} \in \mathcal{I}_{2n_{j}}$ then 
\begin{equation}\label{eqn:composition}
\begin{split}
Z_{R\circ_{i} T}(a_{1}\!\otimes\cdots\otimes &a_{i-1}\!\otimes\!b_{1}\!\otimes \cdots\otimes\!b_{m'}\!\otimes\!a_{i+1}\!\otimes \cdots \otimes\!a_{m}) \\
&:= Z_{R}(a_{1}\!\otimes \cdots \otimes\!a_{i-1}\!\otimes\!Z_{T}(b_{1}\!\otimes \cdots \otimes\!b_{m'})\!\otimes\!a_{i+1}\!\otimes \cdots \otimes\! a_{m})
\end{split}
\end{equation}

\noindent In section \ref{sec:properties} we prove

\begin{thm}[Composition]\label{thm:introcomp}
Let $T$ and $R$ be planar diagrams so that $R \circ_{i} T$ is defined. Then the partition map $Z_{R\circ_{i} T}$ equals the map $Z_{R} \circ_{i} Z_{T}$
\end{thm}

\noindent In addition,

\begin{thm}
The maps $Z_{P}$ have the following properties:
\begin{enumerate}
\item \label{item:tempLieb}{\rm (}Temperley-Lieb property{\rm )} If $P$ has a circle component embedded in the interior of $\Sigma_{\mathbb{D}}$ then $$Z_{P} = (q+q^{-1}) Z_{P'}$$ where $P'$ is the planar diagram found by removing this component,
\item \label{item:conjugation}{\rm (}Conjugation property{\rm )} For $p \in \Z[q^{1/2}, q^{-1/2}]$ let the conjugate $p^{\ast}$ be the polynomial $p(q^{-1})$. There is an involution on $\cleaved{n}$ which allows us to extend to a morphism $\xi \lra \xi^{\ast}$ on $\mathcal{I}_{2n}$. For this map $$Z_{P}(\xi^{\ast}) = Z_{P}(\xi)^{\ast}$$
\end{enumerate}
\end{thm}

\noindent In the language of \cite{Jon2} these properties show that the planar algebra is a unital, spherical, $\ast$-planar algebra. It is unital since each planar diagram for a disc configuration $(\lefty{D}_{0})$ gives rise to a map $Z_{P}$ which satisfies the composition requirement, \cite{Jon2}. It is spherical since $\mathrm{dim}\,\mathcal{I}_{0} = 1$ (over $\Z[q^{1/2},q^{-1/2}]$) and the maps are invariant under isotopies. Finally, the conjugation $\xi \lra \xi^{\ast}$ acts correctly for it to be a $\ast$-planar algebra.\\
\ \\
\noindent Furthermore, it is non-degenerate in the sense that a standard pairing considered in \cite{Jon3} is non-degenerate on each $\mathcal{I}_{2n}$, see section \ref{sec:nondeg}. However, the pairing is not positive definite, as usually desired in the planar algebra approach to studying subfactors, \cite{Jon3}.\\
\ \\
\noindent Following the terminology of planar algebras, as in \cite{Jon1}, we will now call $Z_{P}$ the partition map for $P$.\\
\ \\
\noindent In the next section \ref{sec:properties} we justify the properties cited above. In section \ref{sec:reduction} we describe how this can be extended throught the usual Kauffman state summations to give invariants of tangles which generalize the Jones polynomial and compose well. In section \ref{sec:examples} we provide some example computations. Finally, in section \ref{sec:relationships}, we compare this construction to the Temperley-Lieb planar algebra and the decategorification of Khovanov's tangle homology.

\section{Properties of the partition maps for planar diagrams}\label{sec:properties}

\noindent We now verify the Temperley-Lieb, Conjugation, and Composition properties for the partition maps of planar diagrams.  

\subsection{Properties allowing simplifications of the diagram:}

\begin{prop}[Temperley-Lieb Property]
Let $P$ be a planar diagram and $C$ be a circle component of $P$. Let $P'$ be the planar diagram after deleting $C$ from $P$,  then $Z_{P} = (q + q^{-1}) Z_{P'}$
\end{prop}

\begin{proof} There is a $2\!:\!1$ map $\mathcal{M}(P) \lra \mathcal{M}(P')$ which takes $(M,\sigma)$ to $(M',\sigma')$, where $(M',\sigma')$ has the same configuration of discs, marked points, and arcs as $M$, but the circle $C$ has been removed. Under this map $\sigma'$ is the restriction of $\sigma$ to the remaining circles. The pre-image of $(M',\sigma')$ consists of a two element set $\{M'_{\pm}\}$ where we have added $C$ back with decoration $\sigma(C) = \pm$. If we partition $\mathcal{M}(P)$ into $\mathcal{M}_{+}(P) \sqcup \mathcal{M}_{-}(P)$, based on the value of $\sigma(C)$, then $Z_{P}(\xi)$ can similarly be decomposed. From the definition of $Z_{P}$ the term corresponding to $\mathcal{M}_{+}(P)$ is seen to be $q\,Z_{P'}(\xi)$ while the other is $q^{-1}\,Z_{P'}(\xi)$.
\end{proof}

\begin{cor}
Let $P$ be a planar diagram subordinate to $\mathbb{D} = (\lefty{D}_{0}; \lefty{D}_{1})$ with signaturw $(0;0)$. Then $P$ consists solely of some number $N$ of circles, and $Z_{P}$ is the map \\ $\Z[q^{1/2},q^{-1/2}] \lra \Z[q^{1/2}, q^{-1/2}]$ which multiplies by $(q+q^{-1})^N$.
\end{cor}

\noindent Note that $(q+q^{-1})^{N}$ is also the Jones polynomial of such a tangle, using Khovanov's conventions from \cite{Khov}.\\
\ \\
\noindent In addition, any disc $\lefty{D}_{i}$ with $n_{i}=0$ can be ``removed:''

\begin{prop}\label{prop:nointersect}
Suppose $P$ is a planar diagram subordinate to $\mathbb{D}=(\lefty{D}_{0};\lefty{D}_{1}, \ldots, \lefty{D}_{m})$ with $m \geq 1$. Suppose $\lefty{D}_{i}$ has $n_{i} = 0$ for some $i \geq 1$. Let $P'$ be planar diagram determined by $P$ and subordinate to $\mathbb{D}' = \mathbb{D}\backslash\{D_{i}\}$, with the inherited order. Then $Z_{P'}$ equals $Z_{P}$ under the canonical isomorphism  $\mathcal{I}_{2n_{1}}\otimes \cdots \otimes \mathcal{I}_{2n_{i}} \otimes \cdots \otimes \mathcal{I}_{2n_{m}}$ with $\mathcal{I}_{2n_{1}}\otimes \cdots \otimes \mathcal{I}_{2n_{i-1}} \otimes \mathcal{I}_{2n_{i+1}} \otimes \cdots \otimes \mathcal{I}_{2n_{m}}$ 
\end{prop}

\begin{proof} Since $n_{i} = 0$, $\mathcal{I}_{2n_{i}} = \Z[q^{1/2}, q^{-1/2}]$ in the tensor product $\mathcal{I}_{2n_{1}}\otimes \cdots \otimes \mathcal{I}_{2n_{m}}$, thereby establishing the canonical isomorphism in the statement. The image of each generator is determined only by the configuration of circles in multiply cleaved links in $\mathcal{M}(P)$. As these are the same in the two diagrams, the partition maps $Z_{P}$ and $Z_{P'}$ are identical, under the isomorphism.
\end{proof}

\begin{cor}
 If $P$ is the empty tangle, {\rm (}i.e. it has no components {\rm )} then both $\mathcal{I}_{2n_{1}} \otimes \cdots \otimes \mathcal{I}_{2n_{m}}$ and $\mathcal{I}_{2n_{0}}$ are isomorphic to $\Z[q^{1/2}, q^{-1/2}]$, and $Z_{P}$ is the identity map $\Z[q^{1/2}, q^{-1/2}] \lra \Z[q^{1/2}, q^{-1/2}]$ under the isomorphisms.
\end{cor}

\subsection{Partition maps and conjugation:}

\noindent Multiply cleaved links admit an involution we will call {\em conjugation}:

\begin{defn}
Let $(M,\sigma)$ be a decorated, multiply cleaved link. The conjugate $(M,\sigma)^{\ast}$ is the decorated, multiply cleaved link $(M,\sigma')$ where $\sigma'(C) = \pm$ if and only if $\sigma(C) = \mp$. 
\end{defn}

\noindent We now consider the effect of the conjugation $(M,\sigma) \lra (M,\sigma)^{\ast}$ on the partition map. 

\begin{defn}
For $p \in \Z[q^{1/2}, q^{-1/2}]$,  $p^{\ast}(q) = p(q^{-1})$ is the {\em conjugate} of $p$. 
\end{defn}

\noindent The following proposition follows directly from the definitions.

\begin{prop}
Let $(M,\sigma)$ be a decorated multiply cleaved link. Then $\textsc{Sign}(M,\sigma)^{\ast} = \textsc{Sign}(M,\sigma)$, and $W\big((M,\sigma)^{\ast}\big) = W(M,\sigma)^{\ast}$ in $\Z[q^{\pm 1/2}]$. 
\end{prop}

\noindent We extend these conjugate maps to the tensor products $\mathcal{I}_{2n_{1}} \otimes \cdots \mathcal{I}_{2n_{m}}$ by linearly extending
$$
p(q)\, I_{(L_{1}, \sigma_{1})} \otimes \cdots \otimes I_{(L_{m}, \sigma_{m})} \lra  p^{\ast}(q)\, I_{(L_{1}, \sigma_{1})^{\ast}} \otimes \cdots \otimes I_{(L_{m}, \sigma_{m})^{\ast}}
$$

\begin{prop}[Conjugation property]
Let $P$ be a planar diagram with signature $\textsc{Sign}(P)=(n_{0}; n_{1}, \ldots, n_{m})$, and $\xi \in \mathcal{I}_{2n_{1}} \otimes \cdots \otimes \mathcal{I}_{2n_{m}}$, then $Z_{P}(\xi^{\ast}) = \big(Z_{P}(\xi)\big)^{\ast}$. 
\end{prop} 

\begin{proof} Let $\xi = I_{(L_{1}, \sigma_{1})} \otimes \cdots \otimes I_{(L_{m}, \sigma_{m})}$ be a generator of $\mathcal{I}_{2n_{1}} \otimes \cdots \otimes \mathcal{I}_{2n_{m}}$ and $(M,\sigma)$ be a decorated multiply cleaved link in $\mathcal{M}(P)$ with $\partial_{i}(M,\sigma) = (L_{i}, \sigma_{i})$ for $i \geq 0$. 
 Then $(M,\sigma)^{\ast}$ has $\mathfrak{P}(M^{\ast}) =P$. Thus conjugation is a bijection on $\mathcal{M}(P)$. In addition, $\partial_{i}(M,\sigma)^{\ast} = (L_{i}, \sigma_{i})^{\ast}$ for $i \geq 0$. Since $W((M,\sigma)^{\ast}) = W(M,\sigma)^{\ast}$ in $\Z[q^{1/2}, q^{-1/2}]$
\begin{equation*}
\sum_{\footnotesize\begin{array}{c} (M,\sigma) \in \mathcal{M}(P)\\ \partial_{i}(M,\sigma) = (L_{i},\sigma_{i})^{\ast}, i \geq 1  \end{array}\normalsize} W(M,\sigma)I_{\partial_{0}(M,\sigma)}
= \sum_{\footnotesize\begin{array}{c} (M',\sigma') \in \mathcal{M}(P) \\ \partial_{i}(M',\sigma') = (L_{i},\sigma_{i}), i \geq 1  \end{array}\normalsize} W(M',\sigma')^{\ast}I_{\partial_{0}(M',\sigma')^{\ast}}
\end{equation*}
An thus, by definition, $Z_{P}(\xi^{\ast})=\big(Z_{P}(\xi)\big)^{\ast}$.
\end{proof}

\subsection{Non-degeneracy}\label{sec:nondeg}

\begin{center}
\begin{figure}
\includegraphics[scale=0.5]{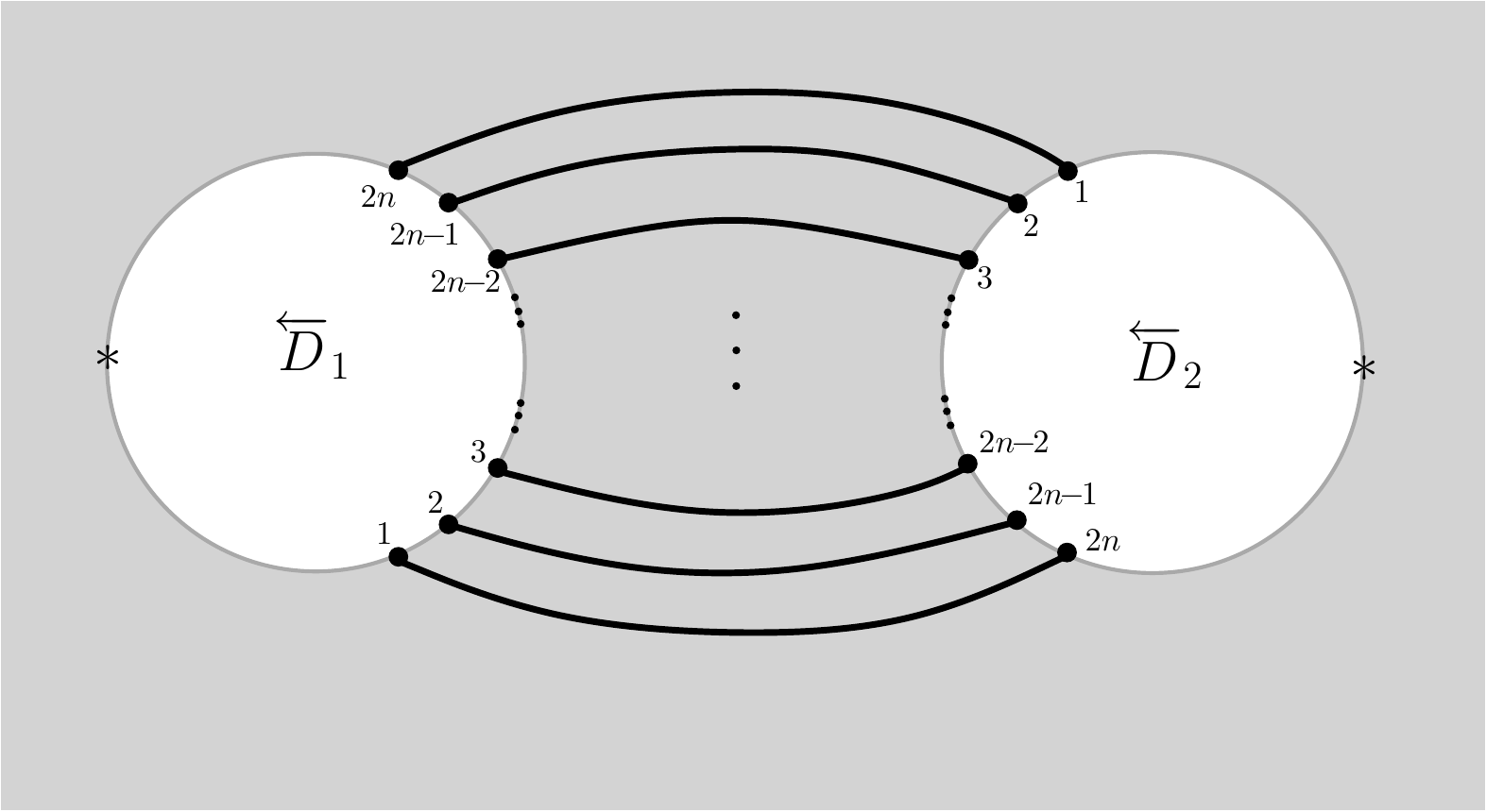}
\caption{}\label{fig:pairing}
\end{figure}
\end{center}

\noindent In this section we show that the planar algebra is non-degenerate. In fact, the pairing depicted defined by the diagram in  Figure \ref{fig:pairing} is a sum of hyperbolic pairs.

\begin{defn}
Let $(L,\sigma) \in \cleaved{n}$ be subordinate to $(\lefty{D}$. Then the {\em dual} $\overline{(L,\sigma)}$ of $(L,\sigma)$ is the cleaved link subordinate to $(\righty{D})$ obtained by changing the orientation on the sphere, but fixing the ordering of the points $L \cap \partial \lefty{D}$.
\end{defn}

\noindent For $L, L' \in \cleaved{n}$, let $\langle L, L' \rangle$ be the image of $L \otimes L'$ under the map $Z_{P}: \mathcal{I}_{2n} \otimes \mathcal{I}_{2n} \lra \Z[q^{1/2}, q^{-1/2}]$ for the diagram $P$ in Figure \ref{fig:pairing}.

\begin{prop}[Non-degeneracy]
Let $L, L' \in \cleaved{n}$, then 
$$
\langle L, L' \rangle = \left\{\begin{array}{cl} 0 & L' \neq \overline{L} \\ 1 & L' = \overline{L} \end{array} \right.
$$
\end{prop}

\begin{proof}
 To see this we compute the partition map for the diagram in Figure \ref{fig:pairing}. The diagrams in $\mathcal{M}(P)$  are obtained by picking two inside matchings $\lefty{m}_{1}$ and $\lefty{m}_{2}$ and filling in the discs $\lefty{D}_{i}$ with them. Then $\partial_{1}(P \circ_{1} \lefty{m}_{1} \circ_{2} \lefty{m}_{2}) = \lefty{m}_{1}  \# \righty{m}_{2}$ where $\righty{m}_{2}$ is the dual of  $\lefty{m}_{2}$, i.e. the same matching considered in an outside disc. On the other hand, $\partial_{2}(P \circ_{1} \lefty{m}_{1} \circ_{2} \lefty{m}_{2}) = \righty{m}_{1} \# \lefty{m}_{2}$. These boundaries are dual, so only dual cleaved links can have non-zero pairing. To see that these pair to give $1$, note that any circle, with any decoration, in the multiply cleaved link necessarily intersects both inside discs. Thus the total weight will be multiples of $q^{1 - 1/2\cdot 2} = 1$ and $q^{-1+1/2\cdot2} = 1$.
\end{proof}

\subsection{Composition of the partition maps for planar diagrams}

\noindent We now verify that the partition maps compose according to the requirements of a planar algebra.

\subsection{Gluing multiply cleaved links:} First, we will define a gluing $\circ_{i}$ for multiply cleaved links. Suppose the multiply cleaved link $(M,\sigma)$ is subordinate to $\mathbb{D}=(\lefty{D}_{0}; \lefty{D}_{1}, \ldots, \lefty{D}_{m})$ and $(N,\nu)$ is subordinate to $\mathbb{D}' = (\lefty{D}'_{0}; \lefty{D}'_{1}, \ldots, \lefty{D}'_{m'})$ with $\partial_{0}(M,\sigma) = \partial_{i}(N,\nu)$. Then $(N,\nu) \circ_{i} (M,\sigma)$ is the (equivalence class) of decorated, multiply cleaved links subordinate to $\mathbb{D}' \circ_{i} \mathbb{D}$ found by removing the disc $\lefty{D}'_{i}$ from the sphere $S_{N}$ and gluing in the disc $\lefty{D}_{0}$ so that $\ast_{\lefty{D}_{0}}$ is glued to $\ast_{\lefty{D}'_{i}}$, and $N \cap \partial \lefty{D}'_{i}$ is identified with $M \cap \lefty{D}_{0}$, in an order preserving manner. The decoration map for $(N,\nu) \circ_{i} (M,\sigma)$ is obtained by restriction of $\nu$ and $\sigma$ to the circles intersecting $\lefty{D}_{0}$ and $\righty{D}_{i}'$. Since  $\partial_{0}(M,\sigma) = \partial_{i}(N,\nu)$ circles which intersect $C = \partial\lefty{D}_{0} =\partial\lefty{D}'_{i}$ receive the same decoration from $\nu$ and $\sigma$. \\
\ \\
\noindent It is straightforward that this provides a well-defined map
\begin{equation*}
\begin{split}
\circ_{i}: \mathcal{M}_{R}(L'_{0};L'_{1},& \ldots,L'_{i-1},L,L'_{i+1}, \ldots, L'_{m'}) \times \mathcal{M}_{P}(L;L_{1}, \ldots, L_{m})\\
& \lra \mathcal{M}_{R\circ_{i} P}(L'_{0};L'_{1}, \ldots,L'_{i-1},L_{1}, \ldots, L_{m},L'_{i+1},\ldots, L'_{m'})
\end{split}
\end{equation*}

\begin{thm}\label{thm:comp}
Let $P$ and $R$ be planar diagram with signatures $\textsc{Sign}(P) = (n_{0}; n_{1}, \ldots, n_{m})$ and $\textsc{Sign}(R) = (n'_{0}; n'_{1}, \ldots, n'_{m'})$, with $n_{0}=n'_{i}$. The partition map $Z_{R\circ_{i} P}$ equals the map $Z_{R} \circ_{i} Z_{P}$ 
\end{thm}

\begin{proof} As both maps are linear over $\Z[q^{1/2},q^{-1/2}]$ it suffices to check the result on generators $$\xi = I_{L'_{1}}\!\otimes \cdots \otimes\!I_{L'_{i-1}}\!\otimes\!I_{L_{1}}\!\otimes\cdots\otimes\!I_{L_{m}}\!\otimes\!I_{L'_{i+1}}\!\otimes\cdots\otimes\!I_{L'_{m'}}$$ 
of the tensor product that is their common domain:
$$\mathcal{I}_{2n'_{1}}\!\otimes\cdots\otimes\!\mathcal{I}_{2n'_{i-1}}\!\otimes\!\mathcal{I}_{2n_{1}}\!\otimes\cdots\otimes\!\mathcal{I}_{2n_{m}}\!\otimes\!\mathcal{I}_{2n'_{i+1}}\!\otimes\cdots\otimes\!\mathcal{I}_{2n'_{m'}}$$
Thus, we will show that $Z_{R\circ_{i} P}(\xi)$ equals $\big(Z_{R} \circ_{i} Z_{P}\big)(\xi)$. 
\begin{equation}\label{eqn:compI}
\big(Z_{R} \circ_{i} Z_{P}\big)(\xi) = Z_{R}(I_{L'_{1}}\!\otimes \cdots \otimes\!I_{L'_{i-1}}\!\otimes\!Z_{P}(I_{L_{1}}\!\otimes\cdots\otimes\!I_{L_{m}})\!\otimes\!I_{L'_{i+1}}\!\otimes\cdots\otimes\!I_{L'_{m'}})\\
\end{equation}
Using linearity and the definitions of $Z_{R}$ and $Z_{P}$, we see that $\big(Z_{R} \circ_{i} Z_{P}\big)(\xi) = \sum_{L' \in \cleaved{n'_{0}}} W_{L'}I_{L'}$ where
\begin{equation}\label{eqn:compIII}
W_{L'} = \sum_{L \in \cleaved{n_{0}}} \sum_{\footnotesize\begin{array}{c}\big((N,\nu), (M,\sigma)\big) \in \mathcal{M}(L';L'_{1},\ldots, L, \ldots, L'_{m}) \times \\ \mathcal{M}(L;L_{1}, \ldots, L_{m})\end{array}\normalsize} W(N,\nu)\cdot W(M,\sigma)
\end{equation}
However, we will show that
\begin{lemma}\label{prop:mclcomp}
The composition $\circ_{i}$ induces a bijection
\begin{equation*}
\begin{split}
\Psi: \bigcup_{L \in \cleaved{n}}\big(\mathcal{M}_{R}(L';&L'_{1}, \ldots,L'_{i-1},L,L'_{i+1}, \ldots, L'_{m'}) \times \mathcal{M}_{P}(L;L_{1}, \ldots, L_{m})\big) \\
&\lra \mathcal{M}_{R\circ_{i} P}(L';L'_{1}, \ldots,L'_{i-1},L_{1}, \ldots, L_{m},L'_{i+1},\ldots, L'_{m'})
\end{split}
\end{equation*}
such that $W\big(N \circ_{i} M) = W(N,\nu)\cdot W(M,\sigma)$.
\end{lemma}

\noindent Then the double sum in equation \eqref{eqn:compIII} is a sum over the elements of the union in lemma \ref{prop:mclcomp}, and, assuming the lemma, we conclude that
$$
W_{L'} = \sum_{\footnotesize\begin{array}{c}(D,\eta) \in \\ \mathcal{M}_{R\circ_{i} P}(L';L'_{1}, \ldots,L'_{i-1},L_{1}, \ldots, L_{m},L'_{i+1},\ldots, L'_{m'})\end{array}\normalsize} W(D,\eta)
$$
 which is the coefficient of $I_{L'}$ in $Z_{R \circ_{i} P}(\xi)$ . Thus, $\big(Z_{R} \circ_{i} Z_{P}\big)(\xi) = Z_{R \circ_{i} P}(\xi)$.\\
\ \\
\noindent We now prove the lemma.

\begin{proof}[Proof of lemma \ref{prop:mclcomp}]

\noindent We start by verifying that $\circ_{i}$ is weight preserving. Suppose $(M,\sigma) \in  \mathcal{M}_{P}(L_{0};L_{1}, \ldots, L_{m})$ and $(N,\nu) \in \mathcal{M}_{R}(L'_{0};L'_{1}, \ldots, L'_{m'})$, with $L_{0}=L'_{i}$. Then
$$
W(N \circ_{i} M) = W(N,\nu) \cdot W(M,\sigma)
$$
in $\Z[q^{1/2}, q^{-1/2}]$. Circles in $\circles{N \circ_{i} M}$ either intersect $C = \partial \lefty{D}_{0}$ or they do not. It is straightforward from the definition of $W(N \circ_{i} M)$ that circles not intersecting $C$ contribute the same factor $\pm q^s$ to $W(N \circ_{i} M)$ as they do to $W(M,\sigma) \cdot W(N,\nu)$. Now let $K$ be a circle which intersects $C$. Suppose $K$ intersects $\lefty{A}$ discs of $R \circ_{i} T$ in $\lefty{D}_{0}$ and $\righty{A}$ discs of $R \circ_{i} T$ in $\righty{D}_{i}'$. The corresponding circle in $(M,\sigma)$ intersects $\lefty{A}+1$ discs: the same $\lefty{A}$ discs in $\lefty{D}_{0}$ and the disc $\righty{D}_{0}$. while the corresponding circle in $(N,\nu)$ intersects $\righty{A}+1$ discs. Thus, $C$ contributes $q^{\pm(1 - \frac{1}{2}(\lefty{A} + \righty{A}))}$ to $W(N \circ_{i} M)$, while the corresponding circles contribute $q^{\pm(1 - \frac{1}{2}(\lefty{A} + 1))}$ to $W(M,\sigma)$ and $q^{\pm(1 - \frac{1}{2}(\righty{A} + 1))}$ to $W(N,\nu)$. Since
$$
q^{\pm(1 - \frac{1}{2}(\lefty{A} + \righty{A}))} = q^{\pm(1 - \frac{1}{2}(\lefty{A} + 1))}q^{\pm(1 - \frac{1}{2}(\righty{A} + 1))}
$$
{\em every} circle in $N \circ M$ contributes the same factor in $W(N \circ_{i} M)$ as in $W(N,\nu) \cdot W(M,\sigma)$, so $W(N \circ_{i} M) = W(N,\nu) \cdot W(M,\sigma)$ . \\
\ \\
\noindent Next we show that the map  $\Psi$ induced by $\circ_{i}$ is a bijection by finding an inverse 
\begin{equation*}
\begin{split}
\Phi : \mathcal{M}_{R\circ_{i} P}&(L'_{0};L'_{1}, \ldots,L'_{i-1},L_{1}, \ldots, L_{m},L'_{i+1},\ldots, L'_{m'}) \\
&\lra \bigcup_{L \in \cleaved{n}}\big(\mathcal{M}_{R}(L'_{0};L'_{1}, \ldots,L'_{i-1},L,L'_{i+1}, \ldots, L'_{m'}) \times \mathcal{M}_{P}(L;L_{1}, \ldots, L_{m})\big)
\end{split}
\end{equation*}

\noindent Let $(D,\eta) \in \mathcal{M}(R \circ_{i} P)(L'_{0};L'_{1}, \ldots,L'_{i-1},L_{1}, \ldots, L_{m},L'_{i+1},\ldots, L'_{m'})$ subordinate to $\mathbb{D}' \circ_{i} \mathbb{D}$ where $\mathbb{D}'=(\lefty{D}_{0}; \lefty{D}'_{1}, \ldots, \lefty{B}, \ldots, \lefty{D}'_{m'})$ and $\mathbb{D} = (\lefty{B};\lefty{D}_{1}, \ldots,\lefty{D}_{m})$. Let $C = \partial \lefty{B}$ and $\ast_{C}$ be the marked point from $\mathbb{D}$ (which is identical to that from $\mathbb{D}'$).
Define $(L,\eta)$ to be the cleaved link with equator $C$, marked point $\ast_{C}$, circles consisting of $\{U \in \circles{N \circ_{i} M}| U \cap C \neq \emptyset\}$, equipped with the decoration map found from restriction from $\eta$. \\
\ \\
Define a multiply cleaved link $(M,\sigma)$ in $S$ by taking $\circles{M} = \{U \in \circles{D}| U \cap \lefty{B} \neq \emptyset\}$ and the decoration $\sigma$ found by restricting $\eta$ to the subset $\circles{M} \subset \circles{D}$. This defines a multiply cleaved link which, given the identification of $\lefty{B}$  with $\lefty{D}_{0}$, is a decorated, multiply cleaved link in $\mathcal{M}(P)$. It is straightforward to verify that $L = \partial_{0}(M,\sigma)$ and $L_{i} = \partial_{i}(M,\sigma)$. Thus, $(M,\sigma) \in \mathcal{M}_{P}(L;L_{1}, \ldots, L_{m})$. Similarly, the subset of $\circles{D}$ defined by $\{U \in \circles{D}| U \cap \righty{B} \neq \emptyset\}$, with the restricted decorated. defines a multiply cleaved link $N$ in $\mathcal{M}_{R}(L'_{0};L'_{1}, \ldots,L,\ldots L'_{m'})$. Therefore,  $\big((N,\nu),(M,\sigma)\big)$ is an element of  
$$
\big(\mathcal{M}_{R}(L'_{0};L'_{1}, \ldots,L'_{i-1},L,L'_{i+1}, \ldots, L'_{m'}) \times \mathcal{M}_{P}(L;L_{1}, \ldots, L_{m})\big)
$$
If we define $\Phi(D,\eta) = \big((N,\nu),(M,\sigma)\big)$ we have a well-defined map as in the statement of the lemma. That $\Phi$ is an inverse of the map $\Psi$ follows from the definitions. 
\end{proof}
\end{proof}

\section{Invariants for oriented tangles}\label{sec:reduction}

\noindent An oriented tangle diagrams $T$ is subordinate to a disc configuration $\mathbb{D}$ if it is drawn in $\Sigma_{\mathbb{D}}$ and intersects each boundary component in an even number of points. Each tangle diagram defines an oriented tangle in $\Sigma_{\mathbb{D}} \times [-1,1]$ up to isotopies which fix the points in $\partial \Sigma_{\mathbb{D}} \times \{0\}$, and any diagrams $T_{1}$ and $T_{2}$ are related by sequences of Reidemeister moves performed in $\mathrm{int}\,\Sigma_{\mathbb{D}}$ if and only if they define isotopic tangles. We will denote the number of positive/negative crossings in $T$ (with reference to the orientation on $\Sigma_{\mathbb{D}}$) by $n_{\pm}(T)$.  \\
\ \\
\noindent We can construct a partition map $Z_{T}$ for a tangle diagram using the state sum approach to the Kauffman bracket. In particular $Z_{T}$ is a linear combination of the maps for the resolution diagrams of $T$.\\ 

\begin{defn}\label{defn:resolutions}
A resolution $\rho$ of $T$ is a map $\rho: \cross{T} \lra \{0,1\}$. For each resolution, $\rho$, there is a planar diagram, $\rho(T)$, called a resolution diagram, obtained by locally replacing {\rm (}disjoint{\rm )} neighborhoods of the crossings of $T$ using the following rule for each crossing $c \in \cross{T}$:\\
$$
\inlinediag[0.3]{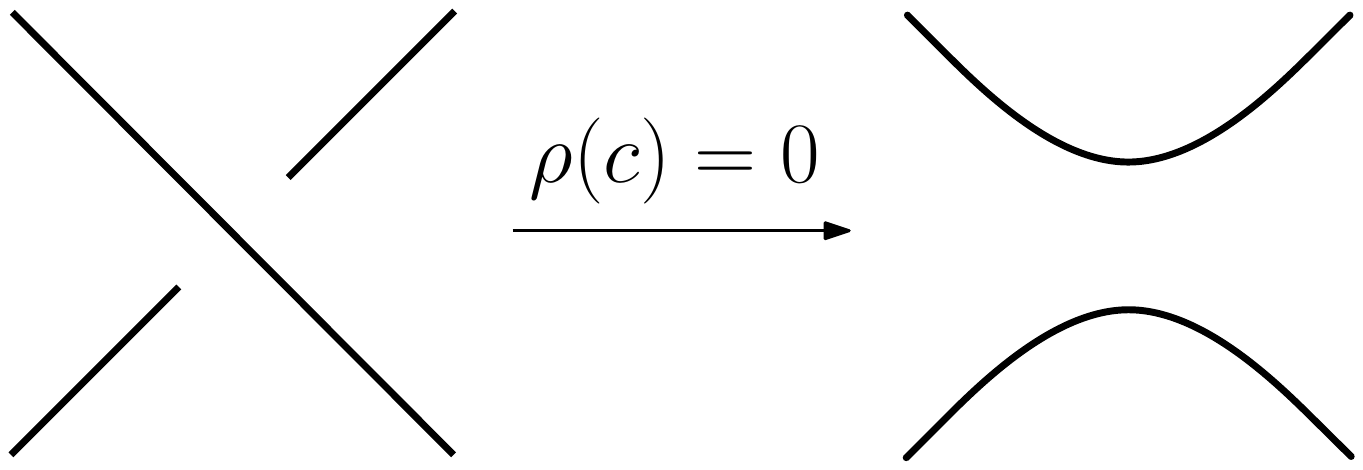} \hspace{1in} \inlinediag[0.3]{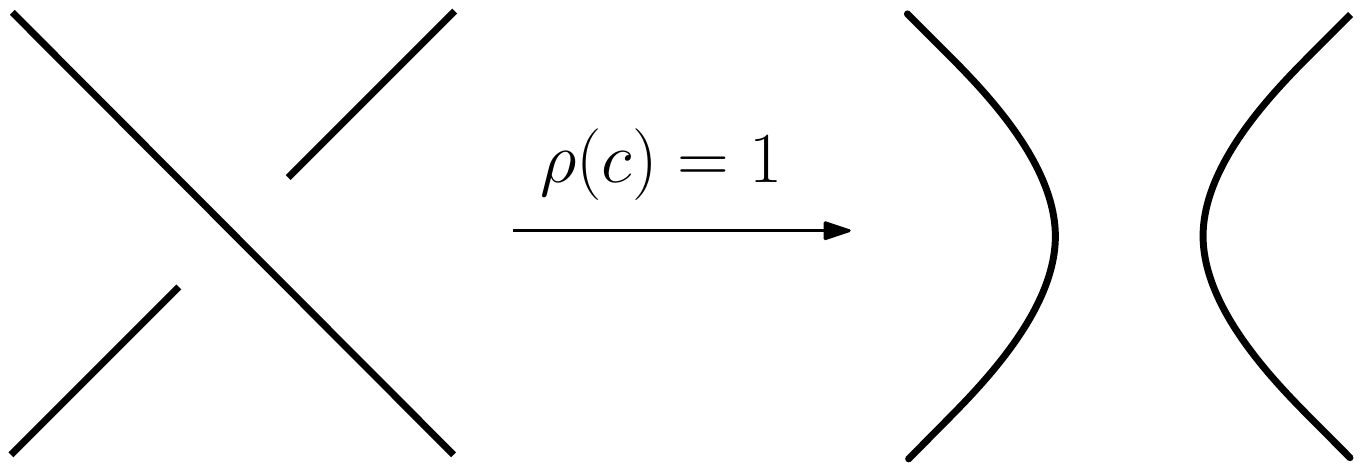} 
$$ 
\noindent The set of resolutions will be denoted $\resolution{T}$. The {\rm (}unshifted{\rm )} homological dimension of $\rho \in \resolution{T}$ is 
$$h(\rho) = \sum_{c \in \cross{T}} \rho(c)$$
\end{defn}
\ \\
\noindent We then define $Z_{T}$ by using the partition maps for the resolutions:\\

\begin{defn}\label{defn:unpart}
Let $T$ be an oriented tangle diagram subordinate to a disc configuration $\mathbb{D}$ with signature $\textsc{Sign}(T)=(n_{0}; n_{1}, \ldots, n_{m})$. The partition map 
$$
Z_{T} :  \mathcal{I}_{2n_{1}}\!\otimes\!\mathcal{I}_{2n_{2}}\!\otimes \cdots \otimes\!\mathcal{I}_{2n_{m}} \lra \mathcal{I}_{2n_{0}}
$$
is defined to be the map 
\begin{equation}
Z_{T} = (-1)^{n_{-}(T)}q^{(n_{+}-2n_{-})(T)} \widetilde{Z}_{T}
\label{eqn:extension}
\end{equation}
where the unshifted partition map $\widetilde{Z}_{T}$ is defined by the state sum
$$
\widetilde{Z}_{T} = \sum_{\rho \in \resolution{T}} (-q)^{h(\rho)} Z_{\rho(T)}
$$
\end{defn}

\noindent Since $Z_{P}$ satisfies the Temperley-Lieb property (\ref{item:tempLieb}) standard linear algebra shows that $Z_{T}$ does as well. Composition of tangle diagrams is defined identically to that for planar diagrams with the added requirement that when $T$ and $R$ are oriented tangle diagrams,  $R \circ_{i} T$ has an orientation which restricts to the original orientations on $T$ and $R$. Standard linear algebra and bookkeeping of resolutions/crossing numbers then implies

\begin{prop}
Let $T$ and $R$ be oriented tangles subordinate to $\mathbb{D}_{T}$ and $\mathbb{D}_{R}$  such that $R \circ_{i} T$ is defined as an oriented tangle. Then 
$$
Z_{R \circ_{i} T} = Z_{R} \circ_{i} Z_{T}
$$
\end{prop} 


\noindent Similarly the conjugation property  $Z_{P}(\xi^{\ast}) = Z_{P}(\xi)^{\ast}$ for planar diagrams, along with resolution/crossing number considerations, implies 

\begin{prop}
Suppose $T$ is a tangle diagram subordinate to $\mathbb{D}$, and let $T^{\ast}$ be the mirror tangle diagram subordinate to $\mathbb{D}$. Then $Z_{T^{\ast}}(\xi^{\ast}) = \big(Z_{T}(\xi)\big)^{\ast}$
\end{prop}


\noindent In addition,  that unshifted partition maps $\widetilde{Z}_{T}$, while {\em not} isotopy invariants of the tangle, do satisfy a skein relation:

\begin{prop}[Unshifted skein relations]\label{prop:skein}
Let $T$ be an unoriented tangle diagram. Suppose $c \in \cross{T}$, and let $T_{0}$ and $T_{1}$ be the tangle diagrams obtained from performing the $0$ and $1$ resolutions locally at $c$. Then 
$$
\widetilde{Z}_{T} = \widetilde{Z}_{T_{0}} - q \cdot \widetilde{Z}_{T_{1}} 
$$
\end{prop}

\begin{proof}[Proof sketch.] Group the terms in the summation in Definition \ref{defn:unpart} into a summand over resolutions with with $\rho(c)=0$ and one with resolutions with $\rho(c) = 1$. The former summand equals $\widetilde{Z}_{T_{0}}$ while the second summand is $-q$ times $\widetilde{Z}_{T_{1}}$. 
\end{proof}  

\noindent Using this skein relation, the Temperley-Lieb, and  the composition properties for $Z_{T}$, it then follows that $Z_{T}$ is an isotopy invariant of the oriented tangles:

\begin{thm}\label{thm:invariance}
If $T_{1}$ and $T_{2}$ are tangle diagrams subordinate to a disc configuration $\mathbb{D}$, and which differ by a Reidemeister move in a disc in the interior of $\Sigma_{\mathbb{D}}$, then $Z_{T_{1}} = Z_{T_{2}}$.
\end{thm}

\noindent To prove Theorem \ref{thm:invariance}, we use the composition property to restrict to the disc where the Reidemeister move occurs and perform a local calculation. Since $\widetilde{Z}_{T}$ satisfies a skein relation and $Z_{U}$ is multiplication by $q + q^{-1}$, the usual diagrammatic arguments for invariance of the Jones polynomial apply.\\
\ \\
\noindent Furthermore, the partition maps $Z_{T}$ generalize Khovanov's version of the Jones polynomial, \cite{Khov}. Let $L$ be an oriented link diagram. Khovanov's version of the Jones polynomial is defined by 
\begin{equation}\label{eqn:Jones}
J_{L}(q) = (-1)^{n_{-}(L)}q^{(n_{+}-2n_{-})(L)}\widetilde{J}_{L}(q)
\end{equation}
where $\widetilde{J}_{L}(q)$ satisfies the following (unoriented) skein relations
$$
\begin{array}{l}
\widetilde{J}_{U}(q) = q + q^{-1} \\
\widetilde{J}_{L \sqcup U}(q) = (q + q^{-1}) \widetilde{J}_{L} \\
\widetilde{J}_{L} = \widetilde{J}_{L_{0}} - q \widetilde{J}_{L_{1}} \\
\end{array}
$$ 
where $U$ denotes an unknot, $L \sqcup U$ is a link with an unlinked and unknotted component $U$, and $L_{0}$ and $L_{1}$ are the $0$ and $1$ resolutions at a crossing of $L$ (see definition \ref{defn:resolutions}). This version of the Jones polynomial arises as the {\em decategorification} of Khovanov homology, \cite{Khov}.

\begin{thm}[Normalization Property]\label{item:normalization}
 If $T$ has signature $(0;0)$ then $T$ comes from a link diagram $L$, and $Z_{T}: \Z[q^{1/2}, q^{-1/2}] \lra \Z[q^{1/2}, q^{-1/2}]$ is multiplication by $J_{L}(q)$.
\end{thm}

\noindent This  follows from the Temperley-Lieb property of the partition maps for planar diagrams. Given $T$ with signature $(0;0)$, any resolution diagram $\rho(T)$ consists of closed circles embedded in $\Sigma_{\mathbb{D}}$. Each map $Z_{\rho(T)}$ is just multiplication by $(q+q^{-1})^{n_{C}(\rho)}$ where $n_{C}(\rho)$ is the number of circles. Substituting into equation (\ref{eqn:extension}) shows that $Z_{T}$ is multiplication by 
$$
(-1)^{n_{-}(T)}q^{(n_{+}-2n_{-})(T)} \sum_{\rho \in \resolution{T}} (-q)^{h(\rho)} (q+q^{-1})^{n_{C}(\rho)}
$$
which is the state sum representation of $J_{T}(q)$. \\

\section{Examples}\label{sec:examples}

\subsection{Computing $Z_{P}$ for a simple planar diagram}\ \\ 

\noindent  Suppose we have a planar diagram $P$ as on the left of 
\begin{center}
\includegraphics[scale=0.6]{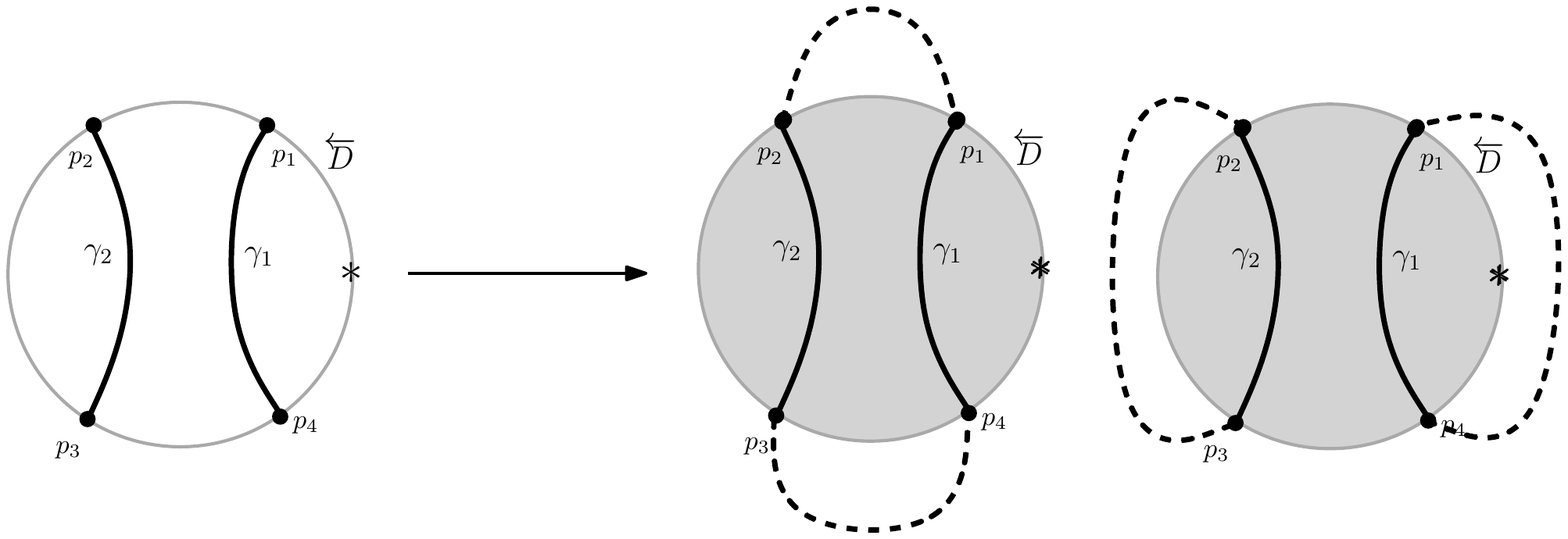}
\end{center}
For this planar diagram the map $Z_{P}$ should be a map $\mathcal{I}_{0} \cong \Z[q^{1/2}, q^{-1/2}] \lra \mathcal{I}_{4}$.\\
\ \\
\noindent There are two outside matchings $\righty{m}_{1}$ and $\righty{m}_{2}$ which can be used in $\righty{D}_{0}$ to construct the multiply, cleaved links with associated tangle $P$. These are shown on the right. Then the first diagram on the right, decorated with a $+$, contributes $q^{1/2}I_{A_{+}}$ to the image of $1 \in  \Z[q^{\pm 1/2}] $, since $W = q^{1-1/2}$, while $A_{-}$ contributes $q^{-1+1/2}I_{A_{-}} = q^{-1/2}I_{A_{-}}$. If we use $\righty{m}_{2}$, and decorate the resulting circles, we obtain the generators $B_{\pm\pm}$ from Figure \ref{fig:I4Gens}. $B_{++}$ contributes $q^{1-1/2}q^{1-1/2}I_{B_{++}} = q\,I_{B_{++}}$, while $B_{+-}$ contributes $q^{1-1/2}q^{-1+1/2}I_{B_{+-}} = I_{B_{+-}}$. Likewise, $B_{-+}$ contributes $q^{-1+1/2}q^{1-1/2}I_{B_{-+}} = I_{B_{+-}}$ and $B_{--}$ contributes  $q^{-1+1/2}q^{-1+1/2}I_{B_{--}} = q^{-1}I_{B_{--}}$. Thus the map $Z_{P}$ is the linear extension of
$$
1 \lra q^{1/2}I_{A_{+}} + q^{-1/2}I_{A_{-}} + q\,I_{B_{++}} + I_{B_{+-}} + I_{B_{+-}} + q^{-1}I_{B_{--}}
$$
Now suppose we move $\ast$ to sit between $p_{1}$ and $p_{2}$,  then, according to our convention, we must renumber the points as $\{q_{1}, q_{2}, q_{3}, q_{4}\}$ where $q_{i} = p_{i+1}$. Then $P'$ will be the matching $q_{1} \leftrightarrow q_{2}$ and $q_{3} \leftrightarrow q_{4}$, while $\righty{m}_{1}$ becomes $q_{1} \leftrightarrow q_{4}$ and $q_{2} \leftrightarrow q_{3}$. Using $\righty{m}_{1}$ now gives generators of type $D_{\pm}$ while using $\righty{m}_{2}$ gives  a generator of type $C_{\pm\pm}$. Indeed, $Z_{P'}$ with the new choice of marked point is the map determined by
$$
1 \lra q^{1/2}I_{D_{+}} + q^{-1/2}I_{D_{-}} + q\,I_{C_{++}} + I_{C_{+-}} + I_{C_{+-}} + q^{-1}I_{C_{--}}
$$

\subsection{Representation of the braid group} Let $B_{2n}$ be the braid group on $2n$-strands. Each diagram $D_{\sigma}$ for $\sigma \in B_{2n}$ can be thought of as an oriented tangle diagram subordinate to two discs. If we think of the diagram $D_{\sigma}$ as being in $[0,1] \times [0,1]$, with $2n$ points on $[0,1] \times \{0\}$ and $[0,1] \times \{1\}$, we can glue $(0,x)$ to $(1,x)$ to obtain a diagram in an annulus. We glue two discs into the annulus to obtain a sphere. The points on $[0,1]\times\{0\}$ are labeled $p_{1}, \ldots, p_{2n}$ in the increasing $x$ direction, and likewise for the points on $[0,1] \times\{1\}$. Let $D_{0}$ be the disc glued to the image $[0,1] \times \{0\}$, and $D_{1}$ be the disc glued to the image of $[0,1] \times \{1\}$.  Then the orderings are correct for a tangle subordinate to $(D_{0}; D_{1})$. We will thus view braids as going down the page. For $\ast_{D_{0}}$ we use the image of $(0,0) \sim (1,0)$, while for $\ast_{D_{1}}$ we use the image of $(0,1)\sim(1,1)$.\\
\ \\
\noindent This construction allows us to associate a map $Z_{\sigma}: \mathcal{I}_{2n} \lra  \mathcal{I}_{2n}$ to any braid diagram $\sigma$. Furthermore, if $\sigma$ and $\sigma'$ are braids in $B_{2n}$ then $\sigma' \circ_{1} \sigma$ is the braid found by stacking $\sigma$ on top of $\sigma'$ (i.e. identifying $[0,1]\times\{0\}$ for $\sigma$ with $[0,1] \times\{1\}$ for $\sigma'$ using $(x,0) \sim (x,1)'$). The composition formula shows that $Z_{\sigma'\sigma}=Z_{\sigma'}Z_{\sigma}$ and thus the maps $Z_{\sigma}$ are a representation of $B_{2n}$.  Furthermore, the non-degeneracy property implies that $Z_{e} = \mathrm{Id}$ where $e$ is the trivial braid. \\
\ \\
\noindent  We compute the representation of $B_{2}$. Let $\sigma \in B_{2}$ correspond the the braid diagram with a single right-handed crossing, and both strands oriented down the page. Then $n_{-}(\sigma) = 0$ and $n_{+}(\sigma) = 1$, so $Z_{\sigma} = (-1)^{0}q^{1 - 2\cdot 0}\widetilde{Z}_{\sigma}$. We can use the skein relation to expand $\widetilde{Z}_{\sigma} = \widetilde{Z}_{I} - q\widetilde{Z}_{T'}$ where $T'$ is the tangle consisting of a cup and a cap, see Figure \ref{fig:cupcap}. The latter map can be found from composing the map $\mathcal{I}_{2} \lra \Z[q^{1/2}, q^{-1/2}]$ for a single cup, with the map for a single cap, which is dual to the former map. These both occur in the annulus, but as one of the circles does not intersect either the cup or cap, that circle can be erased. Thus,
$$
Z_{T'}: \left\{\begin{array}{l} I_{+} \stackrel{\mathrm{cup}}{\lra} q^{1/2} \stackrel{\mathrm{cap}}{\lra} q^{1/2}\big(q^{1/2}I_{+} + q^{-1/2}I_{-}\big) = q I_{+} + I_{-} \\
I_{-} \stackrel{\mathrm{cup}}{\lra} q^{-1/2} \stackrel{\mathrm{cap}}{\lra} q^{-1/2}\big(q^{1/2}I_{+} + q^{-1/2}I_{-}\big) = I_{+} + q^{-1} I_{-} \end{array}\right.
$$ 
Converting to matrix notation in the basis $\{I_{+},I_{-}\}$, we have
\begin{equation*}
\begin{split}
\widetilde{Z}_{\sigma} & = \left[\begin{array}{cc} 1 & 0 \\ 0 & 1 \end{array}\right] - q \left[\begin{array}{cc} q & 1 \\ 1 & q^{-1} \end{array}\right] \\
&= \left[\begin{array}{cc} 1 -q^{2} & -q \\ -q & 0 \end{array}\right] 
\end{split}
\end{equation*}
Thus, 
$$
Z_{\sigma} = q\left[\begin{array}{cc} 1 -q^{2} & -q \\ -q & 0 \end{array}\right] = \left[\begin{array}{cc} q -q^{3} & -q^{2} \\ -q^{2} & 0 \end{array}\right]
$$
\ \\
\noindent We can use this computation to calculate $Z_{\sigma^{-1}}$. Since $\sigma^{-1}$ is represented by the mirror diagram subordinate to the same disc configuration, we have $Z_{\sigma^{-1}}(\xi) = \big({Z}_{\sigma}(\xi^{\ast})\big)^{\ast}$
Thus,
$$
Z_{\sigma^{-1}}(I_{+})= \big({Z}_{\sigma}(I_{-})\big)^{\ast} = \big(-q^{2}I_{+})^{\ast} = -q^{-2}I_{-}
$$
while
$$
Z_{\sigma^{-1}}(I_{-})= \big({Z}_{\sigma}(I_{+})\big)^{\ast} = \big((q-q^{3})I_{+}-q^{2}I_{-})^{\ast} = (q^{-1}-q^{-3})I_{-}-q^{-2}I_{+}
$$
In matrix notation, this yields
$$
Z_{\sigma^{-1}} =  \left[\begin{array}{cc} 0 & -q^{-2} \\ -q^{-2} & q^{-1}-q^{-3} \end{array}\right]
$$
It can then be checked directly that $Z_{\sigma}Z_{\sigma^{-}} = \I = \Z_{\sigma^{-1}}Z_{\sigma}$, as required by the Markov moves on braid diagrams. \\
\ \\
\noindent The informed reader will recognize $Z_{\sigma}$ and $Z_{\sigma^{-1}}$ as the representation derived from the Reshetikhin-Turaev invariant from $U_{q}(\mathfrak{sl}_{2})$. However, this cannot persist for $n > 2$ since the Reshetikin-Turaev invariant will give a map between modules with dimension $2^{n}$ over $\Z/2\Z$, which is not the dimension of $\mathcal{I}_{n}$. For instance, $\mathcal{I}_{4}$ has dimension $12$. 

\begin{center}
\begin{figure}
\includegraphics[scale=0.33]{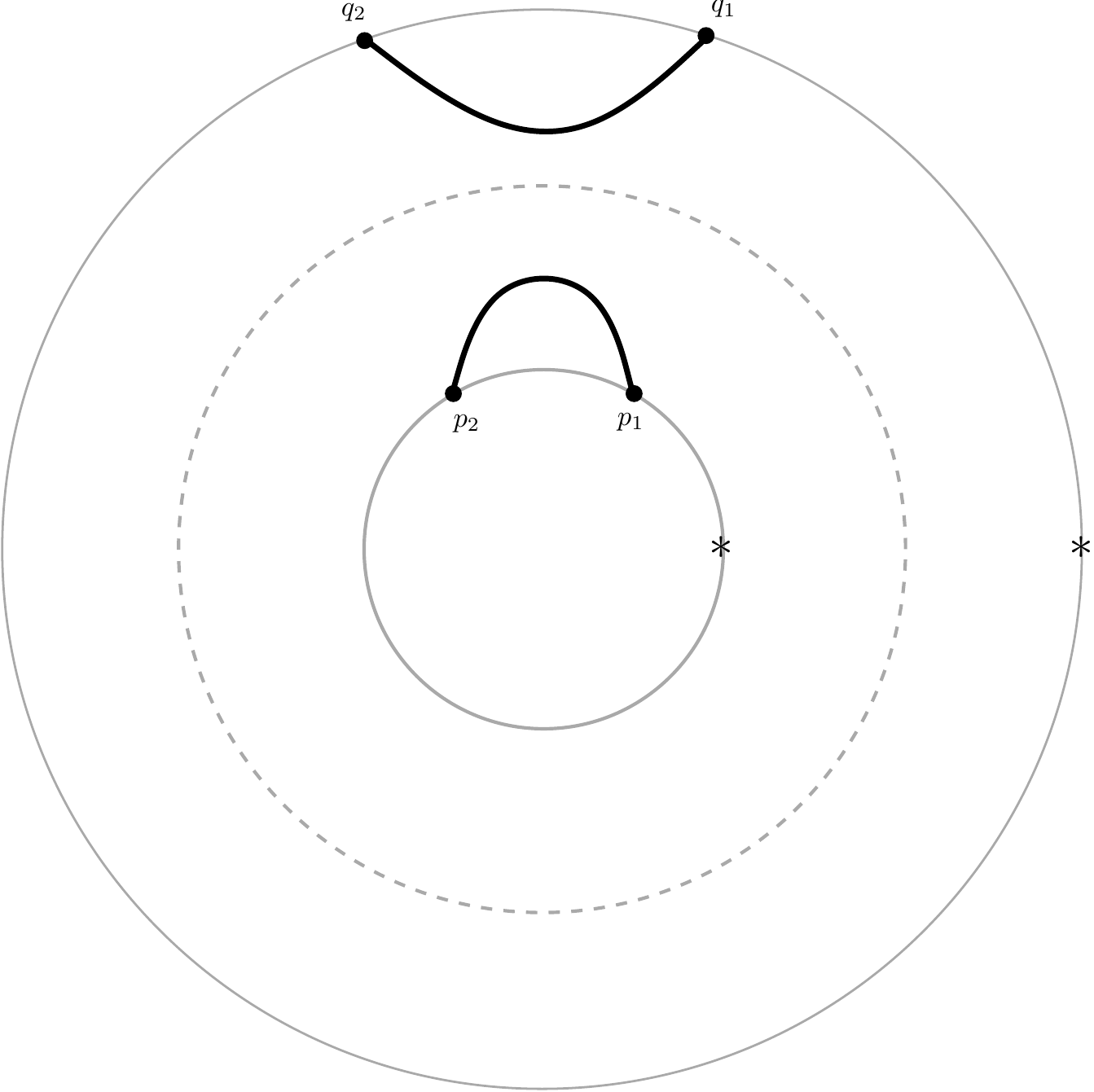}
\caption{} \label{fig:cupcap}
\end{figure}
\end{center}

\noindent{\bf Note:} Above we computed the partition map $\mathcal{I}_{2} \lra \mathcal{I}_{2}$ for the planar tangle in Figure \ref{fig:cupcap}. In matrix form, this map is 
\begin{equation}\label{eqn:cupcap}
\left[\begin{array}{cc} q & 1 \\ 1 & q^{-1} \end{array}\right]
\end{equation}
\noindent If we stack two copies of this tangle, we obtain a tangle with a circle. If we delete this circle we obtain another copy of the tangle in \ref{fig:cupcap}. We illustrate the Temperley-Lieb property by computing the square of the matrix:
\begin{equation*}
\begin{split}
\left[\begin{array}{cc} q & 1 \\ 1 & q^{-1} \end{array}\right]^{2} &= \left[\begin{array}{cc} q & 1 \\ 1 & q^{-1} \end{array}\right]\left[\begin{array}{cc} q & 1 \\ 1 & q^{-1} \end{array}\right]\\
&= \left[\begin{array}{cc} q^{2} + 1 & q+q^{-1} \\ q+q^{-1} & 1 + q^{-2} \end{array}\right]\\
&= (q+q^{-1}) \left[\begin{array}{cc} q & 1 \\ 1 & q^{-1} \end{array}\right]\\
\end{split}
\end{equation*}

\section{Relationship to other theories}\label{sec:relationships}

\subsection{To bordered Khovanov homology}  In \cite{RobD}, \cite{RobA} the author, inspired by bordered Floer homology, describes a bordered Khovanov homology for tangles in a disc. The formal algebraic structure of these tangle invariants mimics the formal structure of \Oz, Lipshitz, and Thurston's description of bordered Floer homology, \cite{Bor1}. In particular, the construction in \cite{RobA} takes a tangle diagram $T$ in a disc $D$, with a marked point $\ast \in \partial D$:\\

\begin{center}
\includegraphics[scale=0.5]{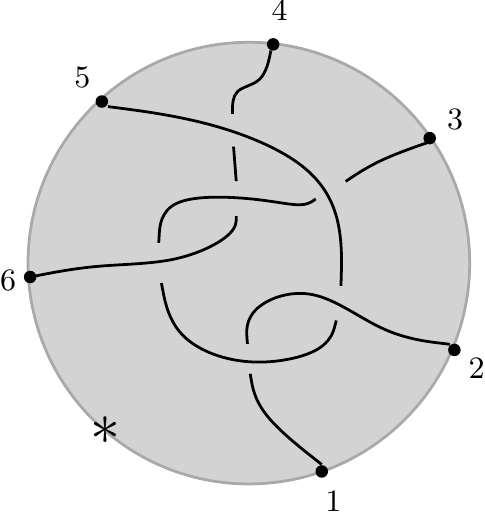}
\end{center}

\noindent and associates to it a bigraded differential module $\leftComplex{T}$ over the differential bigraded algebra $\mathcal{B}\Gamma_{n}$, where $T$ has $2n$ endpoints on $\partial D$. In \cite{Deca} the author describes the decategorification of $\leftComplex{T}$ as a map $[\leftComplex{P}]: \Z[q^{1/2}, q^{-1/2}] \lra \mathcal{I}_{2n_{P}}$ associated to an {\em inside} tangle, i.e. one with signature $(n_{0})$.  Like the Jones polynomial, this decategorification arises as a sum over all the complete resolutions of the diagram. In this paper, we have also assigned a map  $Z_{T}: \Z[q^{1/2}, q^{-1/2}] \lra \mathcal{I}_{2n_{0}}$ to this tangle. 
Similarly, let $T$ be a tangle diagram with signature $(0;n_{1})$, i.e. an {\em outside} tangle. In \cite{Deca} the decategorification of $\rightComplex{T}$ is shown to be a map $[\rightComplex{T}]: \mathcal{I}_{2n_{P}} \lra \Z[q^{1/2}, q^{-1/2}]$ while in this paper we have assigned this diagram a map $Z_{T}:  \mathcal{I}_{2n_{P}} \lra \Z[q^{1/2}, q^{-1/2}]$. These maps are equal:

\begin{prop}
Let $T$ be an inside tangle {\rm (}as in \cite{Deca}{\rm )}. Then the map $[\leftComplex{T}]$ equals the partition map $Z_{T}$, computed from the disc configuration $(\lefty{D})$. If $T'$ is an outside tangle, then $[\rightComplex{T'}]$ equals the partition map $Z_{T'}$, computed using the disc configuration $(\lefty{D}_{0}; \lefty{D})$, where $\lefty{D}_{0}$ is any disk in $\righty{D}$ separated from $T'$.
\end{prop}

\noindent For tangles $T_{1}$ and $T_{2}$ which glue to form a link $L$, the composition of the maps  $[\leftComplex{T_{1}}]$ and $[\rightComplex{T_{2}}]$ is multiplication by $J_{L}(q)$, \cite{Deca}. The analog for the partition maps is proposition \ref{item:normalization}. Thus, the maps in this paper are a generalization of Khovanov's version of the Jones polynomial to tangles over planar surfaces, and their planar algebra structure.\\
\ \\
\noindent Proving this result is a matter of translating between the conventions of this paper with those of \cite{Deca}

\subsection{Comparison with the Temperley-Lieb planar algebra:}\label{sec:comparison} Let $R$ be a ring and $\delta \in R$. The Temperley-Lieb planar algebra $TL_{R}(\delta)$, \cite{Jon2}, is the planar algebra where $TL_{2n}$ is the $R$-module spanned by the planar matchings in a disc with signature $(n)$. Let $P$ be a planar diagram and $m_{1}, \ldots, m_{k}$ be planar matchings in discs with signatures $(n_{1}), \ldots, (n_{k})$. If we glue each $m_{i}$ into the $i^{th}$  boundary of $P$, while aligning marked points, we obtain a planar diagram with $n_{C} \geq 0$ circles. Let $m_{0}$ be the matching resulting from deleting the circles. We then set   
$$
Z_{P}(m_{1}\otimes \cdots \otimes m_{k}) = (q + q^{-1})^{n_{C}}m_{0}
$$
and define the partition map $Z_{P}: TL_{2n_{1}} \otimes\!\cdots\!\otimes TL_{2n_{k}} \lra TL_{2n_{0}}$ to be the linear extension to the tensor product. \\
\ \\
\noindent If we let $R = \Z[q^{1/2}, q^{-1/2}]$ and $\delta = q + q^{-1}$ we obtain a Temperley-Lieb planar algebra with the same properties as the one defined in this paper. However, these are not identical: the dimension of $TL_{2}$ is $1$ since there is only one planar matching on $2$ points, whereas the dimension of $\mathcal{I}_{2}$ is $2$. Likewise, the dimension of $TL_{4}$ is $2$, whereas the dimension of $\mathcal{I}_{4}$ is $12$, see Figure \ref{fig:I4Gens}. This difference of dimension occurs for each $n \geq 1$.\\
\ \\
\noindent There is a relationship between the two, however. According to Jones, \cite{Jon2}, ``each vector space $V_{2n}$ in a unital planar algebra will contain a quotient of the vector space of linear combinations of TL diagrams. '' In our setting, this relationship comes from the observation that for each planar matching $m$ in a disc of type $(n)$, there is an element $Z_{m}(1) \in \mathcal{I}_{2n}$ since $m$ is a planar diagram. Since the planar algebra satisfies the Temperley-Lieb property the image of these elements is closed under the planar algebra compositions, and one obtains a map from $TL_{2n}$ to $\mathcal{I}_{2n}$ which commutes with the planar algebra operations.\\
\ \\
\noindent Likewise, any unital planar algebra provides a representation of the traditional Temperley-Lieb algebra (not planar algebra) of planar $(2n,2n)$-tangles in a square. We consider the map $TL_{2} \lra TL_{2}$ for the diagram in figure \ref{fig:cupcap} representing the sole non-trivial generator of the Temperley-Lieb algebra in this case, and the corresponding map for $\mathcal{I}_{2} \lra \mathcal{I}_{2}$. For $TL_{2}$ this is the map $\Z[q^{1/2}, q^{-1/2}] \lra \Z[q^{1/2}, q^{-1/2}]$ found by multiplying by $q + q^{-1}$. In equation \ref{eqn:cupcap} in the previous section, we computed the map $\mathcal{I}_{2} \lra \mathcal{I}_{2}$ for the generator of the Temperley-Lieb algebra as a matrix $M$. To understand $M$, we first solve the eigenvalue equation $M(v) = \lambda(q) v$. It has two sets of solutions: (1) when $\lambda(q) = q + q^{-1}$ and $v$ is in the submodule generated by $(q^{1/2}, q^{-1/2})$, and (2) when $\lambda(q) = 0$ and $v$ is in the submodule generated by $(-q^{-1/2},q^{1/2})$. Since $Z_{m}(1)=(q^{1/2}, q^{-1/2})$ for the only planar matching on $2$ points, the span of $(q^{1/2}, q^{-1/2})$ is the image of $TL_{2}$. Since $M^2 = (q+ q^{-1})M$ we know that $\mathrm{im}\,M$ also lies in the image of $TL_{2}$. Furthermore, we have seen that $\mathrm{ker}\,M$ is spanned by $(-q^{-1/2},q^{1/2})$.  However, the kernel and the image of $M$ do not span $\mathcal{I}_{2}$ (the matrix with the generating vectors as columns has determinant $q + q^{-1}$ and is thus not invertible). The quotient of $\mathcal{I}_{2}$ by the span is isomorphic to $\Z[q^{1/2}]/(q^2+1)$ as a $\Z[q^{1/2}, q^{-1/2}]$-module. Consequently, the representation on $\mathcal{I}_{2}$ has the image of $TL_{2}$ as an invariant submodule, but the image of $TL_{2}$ is not a direct summand of this representation.   \\
\ \\
\noindent Furthermore we can identify the image of $TL_{4}$ in $\mathcal{I}_{4}$. If we use the basis
\begin{equation}
\big\{I_{C_{++}}, I_{C_{-+}}, I_{C_{+-}},I_{C_{--}},I_{D_{+}},I_{D_{-}},I_{A_{+}},I_{A_{-}},I_{B_{++}},I_{B_{+-}},I_{B_{-+}},I_{B_{--}} \big\}
\end{equation}
\noindent for $\mathcal{I}_{4}$ we can compute the matrices corresponding to the linear maps for the three generators of the Temperley-Lieb algebra on $4$ points. These matrices will be denoted $M_{1}, M_{2},$ and $M_{3}$ as in Figure \ref{fig:matrices}. We can show that the resulting representation is not $TL_{4}$ stabilized by any number of trivial representations by examining the vectors in the intersections of the kernels of these three matrices. Since $\mathcal{I}_{4}$ is $12$ dimensional, and $TL_{4}$ is $2$ dimensional observing a smaller than $12-2 = 10$ dimensional null space for the representation disproves that it is formed by stabilization. The kernels of these three matrices are
\begin{equation*}
\begin{split}
\mathrm{ker}\,M_{1} &= \mathrm{Span}\big\{I_{C_{++}}\!- q^{1/2}I_{A_{+}}, I_{C_{+-}} - q^{-1/2}I_{A_{+}}, I_{C_{-+}}- q^{1/2}I_{A_{-}}, I_{C_{--}} - q^{-1/2}I_{A_{-}}, \\ 
& \hspace{0.75in} I_{D_{+}}, I_{D_{-}}, I_{B_{++}}, I_{B_{+-}}, I_{B_{-+}}, I_{B_{--}} \big\}\\
\end{split}
\end{equation*} 
\begin{equation*}
\begin{split}
\mathrm{ker}\,M_{2} &= \mathrm{Span}\big\{I_{C_{++}}, I_{C_{-+}}, I_{C_{+-}}, I_{C_{--}}, I_{A_{+}}, I_{A_{-}},\\
& \hspace{0.75in} I_{B_{++}}\!- q^{1/2}I_{D_{+}}, I_{B_{+-}} - q^{-1/2}I_{D_{+}}, I_{B_{-+}}- q^{1/2}I_{D_{-}}, I_{B_{--}} - q^{-1/2}I_{D_{-}} \big\}\\
\mathrm{ker}\,M_{3} &= \mathrm{Span}\big\{I_{C_{++}}\!- q^{1/2}I_{A_{+}}, I_{C_{-+}} - q^{-1/2}I_{A_{+}}, I_{C_{+-}}- q^{1/2}I_{A_{-}}, I_{C_{--}} - q^{-1/2}I_{A_{-}}, \\ 
& \hspace{0.75in} I_{D_{+}}, I_{D_{-}}, I_{B_{++}}, I_{B_{+-}}, I_{B_{-+}}, I_{B_{--}} \big\}\\
\end{split}
\end{equation*} 
Notice that $\mathrm{ker}\,M_{1}$ and $\mathrm{ker}\,M_{3}$  are not the same submodules: there is a change in the subscripts. The intersection of these three submodules is seven dimensional:
\begin{equation*}
\begin{split}
\mathrm{ker}\,M_{1} \cap& \mathrm{ker}\,M_{2} \cap \mathrm{ker}\,M_{3} = \\
& \mathrm{Span}\big\{I_{B_{++}}\!- q^{1/2}I_{D_{+}}, I_{B_{+-}} - q^{-1/2}I_{D_{+}}, I_{B_{-+}}- q^{1/2}I_{D_{-}}, I_{B_{--}} - q^{-1/2}I_{D_{-}},\\
& \hspace{0.5in} I_{C_{++}}\!- q^{1/2}I_{A_{+}}, I_{C_{--}} - q^{-1/2}I_{A_{-}}, I_{C_{-+}}+ I_{C_{+-}} - q^{-1/2}I_{A_{+}} - q^{1/2}I_{A_{-}} \big\}
\end{split}
\end{equation*}
We can conclude that $\mathcal{I}_{4}$ is not the direct sum, as a representation, of $TL_{4}$ and trivial representations of the Temperley-Lieb algebra on four points.

\subsection{To Khovanov's functor invariant for tangles:}  In \cite{Khta}, M. Khovanov describes an approach to invariants for traditional $(2n,2m)$-tangles with diagrams in a square and the decategorification of the modules used in his theory. The relationship between the tangle homology in \cite{Khta} and that in \cite{RobD} and \cite{RobA} is the subject of recent work by Andy Manion, \cite{Mani}. He shows that the relationship is novel and quite complicated. We can see a reflection of that observation here. Khovanov's constructions can be modified to give a planar algebra type structure. Furthermore,  in a remark on pg. 714 Khovanov described how to turn the decategorification of his tangle homology into linear maps similar to those examined in this paper. However,
in section 5 of \cite{Khta} he shows that a basis for the decategorification of the modules for tangles in a disc is in one-to-one correspondence with the planar matchings of the tangles boundary points. As a result, the dimensions of the underlying vector spaces will be different from those in this paper, and thus, the maps will not be identical.

\newpage

\begin{figure}
$$
\begin{array}{ccc}
\includegraphics[scale=0.3]{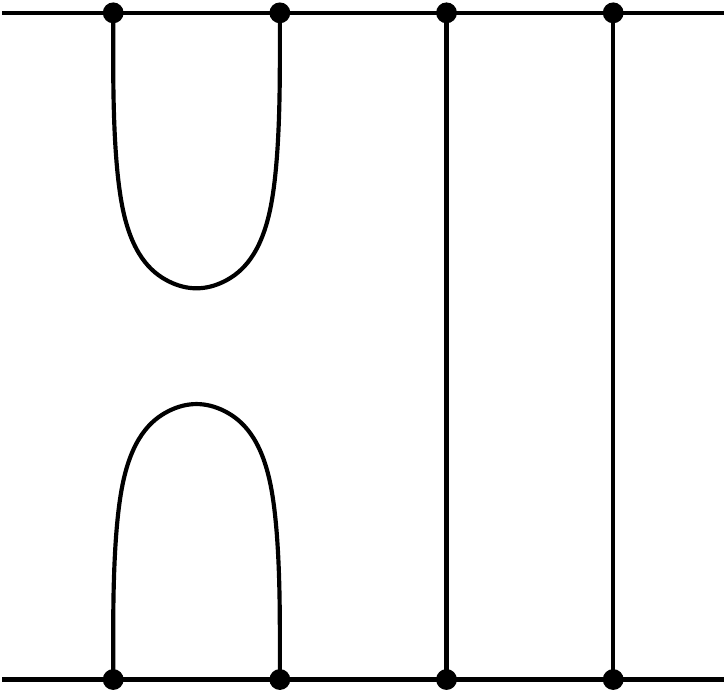} & \hspace{0.5in} &
\left[
\begin{array}{cccccccccccc}
q & 0 & 1 & 0 & 0 & 0 & q^{1/2} & 0 & 0 & 0 & 0 & 0 \\
0 & q & 0 & 1 & 0 & 0 & 0 & q^{1/2} & 0 & 0 & 0 & 0 \\
1 & 0 & q^{-1} & 0 & 0 & 0 & q^{-1/2} & 0 & 0 & 0 & 0 & 0 \\
0 & 1 & 0 & q^{-1} & 0 & 0 & 0 & q^{-1/2} & 0 & 0 & 0 & 0 \\
q^{1/2} & 0 & q^{-1/2} & 0 & 0 & 0 & 1 & 0 & 0 & 0 & 0 & 0 \\
0 & q^{1/2} & 0 & q^{-1/2} & 0 & 0 & 0 & 1 & 0 & 0 & 0 & 0 \\
0 & 0 & 0 & 0 & 0 & 0 & 0 & 0 & 0 & 0 & 0 & 0 \\
0 & 0 & 0 & 0 & 0 & 0 & 0 & 0 & 0 & 0 & 0 & 0 \\
0 & 0 & 0 & 0 & 0 & 0 & 0 & 0 & 0 & 0 & 0 & 0 \\
0 & 0 & 0 & 0 & 0 & 0 & 0 & 0 & 0 & 0 & 0 & 0 \\
0 & 0 & 0 & 0 & 0 & 0 & 0 & 0 & 0 & 0 & 0 & 0 \\
0 & 0 & 0 & 0 & 0 & 0 & 0 & 0 & 0 & 0 & 0 & 0 \\
\end{array}
\right]
\end{array}
$$

$$
\begin{array}{ccc}
\includegraphics[scale=0.3]{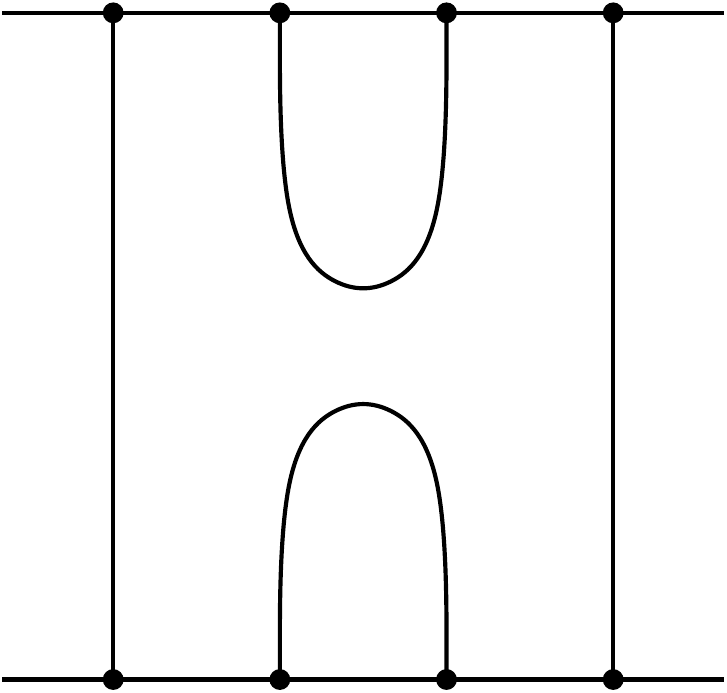} & \hspace{0.5in} &
\left[
\begin{array}{cccccccccccc}
0 & 0 & 0 & 0 & 0 & 0 & 0 & 0 & 0 & 0 & 0 & 0 \\
0 & 0 & 0 & 0 & 0 & 0 & 0 & 0 & 0 & 0 & 0 & 0 \\
0 & 0 & 0 & 0 & 0 & 0 & 0 & 0 & 0 & 0 & 0 & 0 \\
0 & 0 & 0 & 0 & 0 & 0 & 0 & 0 & 0 & 0 & 0 & 0 \\
0 & 0 & 0 & 0 & 0 & 0 & 0 & 0 & 0 & 0 & 0 & 0 \\
0 & 0 & 0 & 0 & 0 & 0 & 0 & 0 & 0 & 0 & 0 & 0 \\
0 & 0 & 0 & 0 & 1 & 0 & 0 & 0 & q^{1/2} & q^{-1/2} & 0 & 0  \\
0 & 0 & 0 & 0 & 0 & 1 & 0 & 0 & 0 & 0 & q^{1/2} & q^{-1/2}   \\
0 & 0 & 0 & 0 & q^{1/2} & 0 & 0 & 0 & q & 1 & 0 & 0 \\
0 & 0 & 0 & 0 & q^{-1/2} & 0 & 0 & 0 & 1 & q^{-1} & 0 & 0  \\
0 & 0 & 0 & 0 & 0 & q^{1/2} & 0 & 0 & 0 & 0 & q & 1    \\
0 & 0 & 0 & 0 & 0 & q^{-1/2} & 0 & 0 & 0 & 0 & 1 & q^{-1} \\
\end{array}
\right]
\end{array}
$$

$$
\begin{array}{ccc}
\includegraphics[scale=0.3]{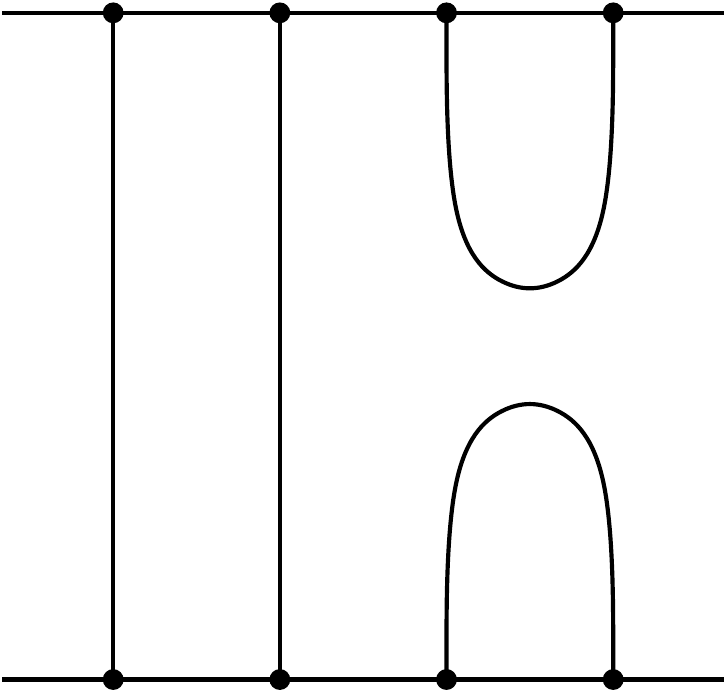} & \hspace{0.5in} &
\left[
\begin{array}{cccccccccccc}
q & 1 & 0 & 0 & 0 & 0 & q^{1/2} & 0 & 0 & 0 & 0 & 0 \\
1 & q^{-1} & 0 & 0 & 0 & 0 & q^{-1/2} & 0 & 0 & 0 & 0 & 0 \\
0 & 0 & q & 1 & 0 & 0 & 0 & q^{1/2}  & 0 & 0 & 0 & 0 \\
0 & 0 & 1 & q^{-1} & 0 & 0 & 0 & q^{-1/2} & 0 & 0 & 0 & 0 \\
q^{1/2} & q^{-1/2} & 0 & 0 & 0 & 0 & 1 & 0 & 0 & 0 & 0 & 0 \\
0 & 0 & q^{1/2} & q^{-1/2} & 0 & 0 & 0 & 1 & 0 & 0 & 0 & 0 \\
0 & 0 & 0 & 0 & 0 & 0 & 0 & 0 & 0 & 0 & 0 & 0 \\
0 & 0 & 0 & 0 & 0 & 0 & 0 & 0 & 0 & 0 & 0 & 0 \\
0 & 0 & 0 & 0 & 0 & 0 & 0 & 0 & 0 & 0 & 0 & 0 \\
0 & 0 & 0 & 0 & 0 & 0 & 0 & 0 & 0 & 0 & 0 & 0 \\
0 & 0 & 0 & 0 & 0 & 0 & 0 & 0 & 0 & 0 & 0 & 0 \\
0 & 0 & 0 & 0 & 0 & 0 & 0 & 0 & 0 & 0 & 0 & 0 \\
\end{array}
\right]
\end{array}
$$
\label{fig:matrices}
\caption{Matrices $M_{1}, M_{2},$ and $M_{3}$ (vertically from top) for the mappings $\mathcal{I}_{4} \lra \mathcal{I}_{4}$ corresponding to the Temperley-Lieb generators shown to the left of each matrix}
\end{figure}

\end{document}